%% file: main.tex
\documentclass[11pt,leqno]{article}
\usepackage{graphicx, amsfonts, amsthm, amsxtra, amssymb, verbatim, makeidx}
\usepackage{subeqnarray, relsize}
\usepackage[mathscr]{euscript}
\usepackage[english]{babel}
\usepackage[fixlanguage]{babelbib}
\usepackage{tikz-cd}
\usepackage[utf8]{inputenc}
\usepackage[english]{babel}

\usepackage{wrapfig}
\usepackage{amssymb, amsmath, amsthm}
\usepackage{graphicx}
\usepackage{color}
\usepackage{amssymb}
\usepackage{url}
\usepackage{pdfpages}
\usepackage{fancyhdr}
\usepackage{subfig}
\usepackage{titlesec}
\usepackage{enumerate}
\usepackage{comment}
\usepackage{bigints}
\usepackage{diagbox}
\usepackage{cite}
\usepackage[fixlanguage]{babelbib}
\usepackage[unicode,
            psdextra,
            colorlinks=true,
            linkcolor=blue,
            citecolor=green
            ]{hyperref}
\usepackage[nameinlink,capitalize,noabbrev]{cleveref}

\makeindex
\makeglossary

\theoremstyle{definition}
\newtheorem{df}{\bf Definition}[section]
\newtheorem{teo}{Theorem}[section]
\newtheorem{lem}{Lemma}[section]

\newtheorem{cor}{Corollary}[section]
\newtheorem{prop}{Proposition}[section]

\newtheorem{rem}{Remark}[section]

\renewenvironment{proof}{{\bfseries \noindent Proof} }{ \qed \\}

\begin{document}

\def\R{\mathbb{R}}                   
\def\Z{\mathbb{Z}}                   
\def\Q{\mathbb{Q}}                   
\def\C{\mathbb{C}}                   
\def\N{\mathbb{N}}                   
\def\uhp{{\mathbb H}}                
\def\A{\mathbb{A}}

\def\cf{r}
\def\T{\sf{T}}
\def\Hess{\text{Hess}}
\def\sing{\text{sing}}
\def\supp{\text{supp}}
\def\CH{{\rm CH}}
\def\P{\mathbb{P}}
\def\Gal{\text{Gal}}
\def\pr{\text{pr}}
\def\GHod{\text{GHod}}
\def\SHod{\text{SHod}}
\def\Hod{\text{Hod}}
\def\res{\text{res}}
\def\prim{\text{prim}}
\def\dR{\text{dR}}
\def\Jac{\text{Jac }}

\def\ker{{\rm ker}}              
\def\GL{{\rm GL}}                \def\HF{{\rm HF}}                
\def\ker{{\rm ker}}              
\def\coker{{\rm coker}}          
\def\im{{\rm Im}}               
\def\coim{{\rm Coim}}            

\def\End{{\rm End}}              
\def\rank{{\rm rank}}                
\def\gcd{{\rm gcd}}                  

\begin{center}
{\LARGE\bf Modular Vector Fields for Lattice Polarized K3
}
\\
\vspace{.25in} {\large {{\sc Walter Andrés Páez Gaviria}}\footnote{
{\tt wapaezg@unal.edu.co}}}
\end{center}


\begin{abstract}
We consider a moduli space of lattice polarized K3 surfaces with the additional information of a frame of the trascendental cohomology with respect to the lattice polarization. This moduli space is proved to be quasi-affine, and the existence of vector fields on it, called modular vector fields, is proved. A purely algebraic version of the algebra of Siegel quasi-modular forms is obtained as the algebra of global regular functions over this moduli space, with a differential structure coming from the modular vector fields. By means of trascendental considerations we are able to obtain a differential algebra of meromorphic Siegel quasi-modular forms from the previous algebra.
\end{abstract}

\section{Introduction}
\label{intro}
\input{intro.tex}

\section{The Clingher-Doran family}
\label{sec2}
\input{s2.tex}

\section{Algebraic De Rham  Cohomology and Gauss-Manin connection}
\label{chapalgcoh}
\input{s3.tex}

\section{Moduli spaces of enhanced $N$-polarized K3 surfaces}
\label{modular}
\input{s4.tex}

\section{The generalized period map and the $\mathsf{t}$-map}
\label{period}
\input{s5.tex}

\section{Modular vector fields and Siegel modular forms}
\label{sec6}
\input{s6.tex}

\section{Appendix}
\label{sec7}
\input{s7.tex}


\bibliographystyle{alpha}

\bibliography{ref}



\end{document}

%% file: intro.tex
Classical modular forms are holomorphic functions defined on the Poincaré half-plane $\mathbb{H}=\{\tau\in\mathbb{C}|\, \text{Im}\, \tau>0\}$, satisfying some functional equations, which, in particular, imply that they are periodic. Besides, a holomorphicity condition at infinity is also required. They form an algebra, usually denoted by $\mathfrak{M}$, generated over $\mathbb{C}$ by the Eisenstein series $E_4$ and $E_6$. 

For the non-expert, the Eisenstein series $E_k:\mathbb{H}\rightarrow \mathbb{C}$, for $k\geq 0$, can be defined by means of
\begin{equation*}E_{k}(\tau)=1-\frac{2k}{B_k}\sum_{n=1}^{\infty} \sigma_{k-1}(n)q^n, \,\,\,\,  q=e^{2\pi i \tau},\, \tau\in \mathbb{H}.
\end{equation*}

\noindent Here $\sigma_i(n)=\sum_{d|n}d^i$,  $B_k$ is the $k$-th Bernoulli number. As in this case, the Fourier coefficients of the classical modular forms are expected to contain "arithmetic" information, whence part of their importance in mathematics. It is worth mentioning that $E_2$ is \textit{not} a modular form.

Going back to the algebra $\mathfrak{M}$ of classical modular forms, since it is finitely generated over $\mathbb{C}$, $Spec(\mathfrak{M})$ is an affine variety over $\mathbb{C}$. This suggests that we could have arrived at the algebra of classical modular forms by having considered the affine variety $Spec(\mathfrak{M})$, or a suitable open subvariety, say $\mathsf{S}$, of it, and taken its algebra of global regular functions. Nonetheless, this procedure would only allow us to only recover the \textit{algebraic} structure of $\mathfrak{M}$, leaving out its \textit{trascendental} structure, i.e., that this is an algebra of holomorphic functions on $\mathbb{H}$. To solve this we would need to find a holomorphic map from $\mathbb{H}$ to $\mathsf{S}$, which would allow us to pullback regular function on $\mathsf{S}$ to holomorphic functions on $\mathbb{H}$. 

It turns out that such a procedure can be accomplished by taking $\mathsf{S}$ to be the moduli space of elliptic curves over $\mathbb{C}$, enhanced with a holomorphic $1$-form, and the holomorphic map from $\mathbb{H}$ to $\mathsf{T}$ can be taken to be essentialy the inverse of the classical period map for elliptic curves. Let us keep in mind this procedure while going on a further disgression. 

It is not hard to see that classical modular forms are not closed under differentiation, but a remarkable fact happens: the derivative of a modular form belongs to the algebra generated by the Eisenstein series $E_2$, $E_4$ and $E_6$. Furthermore, the so-called Ramanujan relations between Eisenstein series

$$E_2'=\frac{E_2^2-E_4}{12},$$
    \begin{equation}E_4'=\frac{E_2 E_4-E_6}{3},\end{equation}
    $$E_6'=\frac{E_2E_6-E_4^2}{2},$$

imply that this algebra is closed under differentiation. This algebra, usually denoted by $\tilde{\mathfrak{M}}$, is called the algebra of quasi-modular forms. The degrees given by the subindices of the Eisenstein series give it a structure of graded algebra. Its homogeneous elements are called quasi-modular forms. Its differential structure is completely determined by the Ramanujan relations between Eisenstein series.

We could ask whether a procedure like the one metioned before for modular forms is possible for the algebra of quasi-modular forms. A positive answer to this question was given in \cite{mov2008,mov2012,pascal}. The elements involved in its construction are the moduli space of elliptic curves over $\mathbb{C}$, enhanced with additional  cohomological data, and a generalized period map. Furthermore, the Ramanujan relations between Eisenstein series can be recovered as an algebraic vector field on the forementioned moduli space. It is remarkable that \textit{moduli} spaces give rise to \textit{modular} forms.

The upper half-plane $\mathbb{H}$ can be generalized to the hyperbolic $n$-space $\mathbb{H}^n$, and in a similar fashion classical modular forms on $\mathbb{H}$ can be generalized to Siegel modular forms on $\mathbb{H}^n$, also called Siegel modular forms of genus $n$. This work is about an extension of the notion of Siegel modular form of genus two, which is an analog of the extension of the classical elliptic modular forms to quasi-modular forms, by means of a similar procedure to the one mentioned above. Its key elements are a moduli space of a certain type of K3 surfaces enhanced with cohomological data, and algebraic vector fields defined on it, which allow us to construct an algebraic version of the graded algebra of Siegel quasi-modular forms. The previous algebraic construction can be used to obtain holomorphic functions on the complement of a fine subset of $\mathbb{H}_2$ by using transcendental considerations. 

\subsection{The GMCD method}\label{1.3}
This work fits into a broader project called \textit{Gauss-Manin Connection in Disguise} (GMCD for short). For a more general and complete presentation of this project, we refer to the main book of its creator \cite{mov20}. The main purpose of the project is to give a geometric framework for developing a theory of quasi-modular forms from it, together with their generating differential equations. In what follows we explain roughly the main ingredients of this project:

\begin{itemize}
\item \textbf{Algebraic part:} Here, for a kind of algebraic varieties (it is worth mentioning that in \cite{vogrinthesis} there is an application of this method to Landau-Ginzburg models, which may not fit into the presentation given here) we construct the moduli space $\mathsf{T}$ or a patch of it, which classifies the varieties enhanced with elements of its middle de Rham cohomology, compatible with the Hodge filtration, and with a fixed intersection matrix; \textit{this space is expected to be a quasi-affine complex variety}. The algebra of algebraic quasi-modular forms in this case correspond to the global regular functions on $\mathsf{T}$. The automorphic data of the previous forms is contained in an algebraic group $\mathsf{G}$ such that the quotient $\mathsf{T}/\mathsf{G}$ is the moduli of the algebraic varieties under consideration. Next, we must prove the existence and uniqueness of algebraic modular vector fields which describes the differential structure of the algebra of quasi-modular forms previously defined.

\item \textbf{Transcendental part:} To obtain holomorphic functions out of the algebra of algebraic quasi-modular forms constructed in the algebraic part, we must consider a classical period domain $\mathsf{D}$ associated to our class of varieties (observe that $\mathsf{D}$ need not be hermitian symmetric, so this procedure allows to construct an automorphic form theory for new domains; the case of mirror quintics, is treated in \cite{mov17}). To construct an automorphic theory on $\mathsf{D}$, we need to construct a $\tau$-map $\mathsf{D}\rightarrow \mathsf{U}$, where $\mathsf{U}$ is a generalized period domain, which is the natural codomain for the period map $\mathcal{P}:\mathsf{T}\rightarrow\mathsf{U}$.  Next we must prove, at least, the local injectivity of $\mathcal{P}$ (if $\mathcal{P}$ is a biholomorphism, we have a possitive answer to the algebraization for $\mathsf{U}$). 

\item \textbf{Mixing-up:} We use the previous functions to construct a $\mathsf{t}$-map $\mathsf{t}:\mathsf{D}\rightarrow \mathsf{T}$ which allows us to pull-back the algebraic quasi-modular forms constructed in the algebraic part, to quasi-modular forms on $\mathsf{D}$. The modular vector fields give us the differential structure.
\end{itemize}

\subsection{Main results}\label{1.4}

In this work we apply the GMCD method to quasi-ample lattice polarized K3 surfaces. 

In the case of a lattice polarization of $N=H\oplus E_8\oplus E_7$, by applying GMCD to $N$-polarized K3 surfaces, we obtain a differential algebra of meromorphic Siegel quasi-modular forms of genus two for the full group $\mathsf{Sp}_4(\mathbb{C})$ defined on $\mathbb{H}_2.$ We proceed to explain the moduli space involved in this construction.

Let $\Psi$ be a matrix
\begin{equation}\label{psi}
\Psi:= \begin{bmatrix}
0 & 0 & 1 \\
0 & \Psi' & 0\\
1 & 0 & 0 
\end{bmatrix},
\end{equation}
where $\Psi'$ is a non-singular and symmetric $3\times 3$ matrix with complex entries. 

\begin{df}\label{tmoduli}
Let us denote by $\mathsf{T}_{\Psi}$ the moduli of tuples $(X,\iota,\alpha_1,\ldots,\alpha_5)$ in which:
\begin{itemize}
\item[i.] $X$ is a smooth complex algebraic $N$-polarized K3 surface;
\item[ii.] $\iota:N\rightarrow H_{dR}^2(X/\mathbb{C})$ is a lattice polarization;
\item[iii.] $(\alpha_1,\ldots,\alpha_5)$ is a basis of $H^2_{dR}(X/\mathbb{C})_{\iota}:=H_{dR}^2(X/\mathbb{C})/\iota(N)$ such that $\alpha_1\in F^2$, $\alpha_1,\ldots,\alpha_4\in F^1$ and $[\langle\alpha_i,\alpha_j \rangle]=\Psi$. Here, $H_{dR}^2(X/\mathbb{C})$ denotes the second algebraic de Rham cohomology group of $X$ over $\mathbb{C}$, and $F^{\bullet}$ denotes its Hodge filtration.
\end{itemize}
\end{df}

The following algebraic group contains the automorphic-data of the Siegel quasi-modular forms:

\begin{df}\label{alggroup} We define the complex algebraic group
\begin{equation*}
\mathsf{G}_{\Psi}:=\{g\in \mathsf{Mat}_{5}(\mathbb{C})\, | \, g^T\Psi g=\Psi\text{ and } g^{T} \text{respects Hodge filtration} \}.
\end{equation*}
\end{df}
Here, we say that $g^T$ respects Hodge filtration if $g^T$ is of the following form
\begin{equation*}
\begin{bmatrix}
*_{1\times 1} & 0 & 0 \\
* & *_{3\times 3} & 0\\
* & * & *_{1\times 1}\\
\end{bmatrix}\in \mathsf{Mat}_5(\mathbb{C}).
\end{equation*}

Since $\mathsf{T}_{\Psi}$ and $\mathsf{G}_{\Psi}$ are independent of $\Psi$ over $\mathbb{C}$, we will denote them by $\mathsf{T}$ and $\mathsf{G}$, respectively.

\begin{teo}[Theorem \ref{quasiaffine}] There is a non-empty patch 
$\mathsf{O}\subset\mathsf{T}_{\Psi}$ such that $\mathsf{O}$ is a quasi-affine complex variety.
\end{teo}

\begin{df}
The algebra of algebraic Siegel quasi-modular forms $\widetilde{\mathfrak{M}}^{alg}(\mathsf{Sp}_4(\mathbb{C}))$ of genus two for the group $\mathsf{Sp}_4(\mathbb{C})$ is defined to be the algebra $H^{0}(\mathsf{O},\mathcal{O}_{\mathsf{O}})$.
\end{df}

The differential structure of the previous algebra comes from the following theorem.

\begin{teo}[Theorem \ref{bigteo}]
For each $\mathfrak{g}\in \mathfrak{sp}_4(\mathbb{C})$, there exists a unique algebraic vector field $\mathsf{R}_{\mathfrak{g}}\in \Theta_{\mathsf{O}}$ such that
\begin{equation}\label{modular1}
\nabla_{\mathsf{R}_{\mathfrak{g}}} \alpha=\mathfrak{g}^T\alpha,
\end{equation}

where $\alpha=(\alpha_1,\ldots,\alpha_5)^T$.
\end{teo}

To pass from these algebraic constructions to transcendental ones we will need the generalized period map, which is basically an extension of the classical period map by considering not only holomorphic two-forms but also the whole middle transcendental de Rham cohomology. In this context, we have 

\begin{teo}\label{tore}[Theorem \ref{bihol}]
The generalized period map $\mathcal{P}:\mathsf{T\rightarrow}\mathsf{Sp}_4(\mathbb{C})\backslash\mathsf{\Pi}$ is a biholomorphism.
\end{teo}

Here $\mathsf{\Pi}$ is a smooth complex manifold where the periods of the transcendental de Rham cohomology live.

We observe that Theorem \ref{tore} holds for any lattice polarization $N$ (under suitable changes).

The final piece of the construction is the definition of the $\mathsf{t}$-map 

\begin{df}
\begin{itemize}\item The $\tau$-map of the Doran and Clingher family is the holomorphic map 
\begin{gather}\tau:\mathbb{H}_2\rightarrow \mathsf{\Pi}, \notag\\
\begin{bmatrix}
\tau_1 & \tau_2 \\
\tau_2 & \tau_3
\end{bmatrix}\mapsto \begin{bmatrix}
\tau_2^2-\tau_1\tau_3 & -\tau_3 & -2\tau_2 & -\tau_1 & 1\\
\tau_3 & 0 & 0 & 1 & 0\\
\tau_2 & 0 & -1 & 0 & 0\\
\tau_1 & 1 & 0 & 0 & 0\\
1 & 0 & 0 & 1 & 1
\end{bmatrix}.\notag
\end{gather}

We observe that the $\tau$-map satisfies that the composition $\mathbb{H}_2\xrightarrow {\tau} \mathsf{\Pi}\rightarrow \mathsf{\Pi}/ \mathsf{G}=\mathbb{H}_2\cup\overline{\mathbb{H}_2}$ is the identity. Here $\mathbb{H}_2$ is the Siegel upper-half space of genus two and the bar represents the conjugation of its elements.
\item The $\mathsf{t}$-map is defined as the composition $\mathbb{H}_2\xrightarrow {\tau} \Gamma\backslash\mathsf{\Pi}\xrightarrow{\mathcal{P}^{-1}}\mathsf{T}$.
\end{itemize}
\end{df}

By mixing-up the previous algebraic and transcendental constructions, we get the definition of the algebra of Siegel quasimodular forms.

\begin{df}
The algebra of Siegel quasimodular forms is defined to be the pullback $\mathsf{t}^*(\mathcal{O}_{\mathsf{O}})$.
\end{df}

\subsection{Organization of the article}

\begin{itemize}
\item \textbf{Section 2:} We recall Clingher and Doran construction of $N$-polarized K3 surfaces.
\item \textbf{Section 3:} Here we introduce the cohomology of a deformation and the basics of tame polynomials theory and the needed computations for arriving at a basis compatible with the Hodge filtration, and the Gauss-Manin connection.

\item \textbf{Section 4:} The moduli space $\mathsf{T}$ of enhanced $N$-polarized K3 surfaces is introduced and its algebra of global regular functions is described. This is defined to be the algebra of algebraic Siegel quasi-modular forms of genus two. The algebraic group $\mathsf{G}$ which contains the automorphic data of the Siegel quasi-modular forms is defined and computed. The AMSY-Lie algebra is defined and proved to be isomorphic to $\mathfrak{sp}_4(\mathbb{C})$. The previous algebraic group and Lie algebra is also computed for arbitrary lattice polarizations. 

\item \textbf{Section 5:} Here we introduced the manifold of period matrices $\mathsf{\Pi}$, which is a $10$-dimensional smooth manifold where the periods of $N$-polarized K3 surfaces live. We also define the generalized period domain $\mathsf{U}$ which takes into account the monodromy action. The generalized period domain $\mathcal{P}:\mathsf{T}\rightarrow \mathsf{U}$ is defined an proved to be a biholomorphism for any lattice polarization. Next, we construct the $\tau$-map for $N$-polarized K3 surfaces and explain how it was obtained. By means of the previous results, we construct the map $\mathsf{t}:\mathbb{H}_2\rightarrow \mathsf{T}.$

\item \textbf{Section 6:} In this section, the uniqueness and existence of algebraic vector fields $\mathsf{R}_{\mathfrak{g}}$ on $\mathsf{T}$ for each $\mathfrak{g}\in\mathfrak{sp}_4(\mathbb{C})$ is proved. The algebras of algebraic Siegel modular and quasi-modular forms of genus two are defined. Next, we prove a characterization of the algebra of algebraic classical Siegel modular forms inside the algebra of algebra of algebraic Siegel quasi-modular forms by means of the previously constructed vector fields. Finally, we use the $\mathsf{t}$-map defined in Section 4 to pullback the algebras of algebraic Siegel modular and quasi modular forms of genus two to algebras of meromorphic functions on $\mathbb{H}_2$.
\end{itemize}

\bigskip

\noindent\textbf{Acknowledgements.} I would like to thank my supervisor, Hossein Movasati. Part of this work was developed during my stay at Hamburg University, visiting professor Murad Alim.

%% file: s2.tex
It is a well-known result that every complex elliptic curves is isomorphic to a projective variety defined by the vanishing of a homogeneous polynomial of the form $$y^2z-4x^3+g_2xz^2+g_3z^3,$$ with $g_2^3-27g_3^2\neq 0$. This polynomial is usually called the \textit{Weierstrass normal form} of the elliptic curve. A normal-form type theorem for $N$-polarized compĺex K3 surfaces was proven in \cite{dor}. More specifically, we have the following theorem:

\begin{teo}[\cite{dor}] Every $N$-polarized $K3$ surface is isomorphic to the minimal resolution of a hypersurface in projective space given as the zero locus of a homogeneous polynomial of the form 
\begin{equation}
F_{a,b,c,d}=wy^2z-4x^3z+3aw^2xz+bw^3z+cwxz^2-\frac{1}{2}(dw^2z^2+w^4).
\end{equation}

Here $a,b,c,d\in \mathbb{C}$, and $c\neq 0$ or $d\neq 0$.
\end{teo}

In this section we give an exposition of these results. They will be used in the following sections, when we consider moduli spaces of enhanced $K3$-surfaces.

Let $Y(a,b,c,d):=\{F_{a,b,c,d}=0\}\subset \mathbb{P}^3,$ and $X(a,b,c,d)$ its minimal resolution. Since it is smooth and complete, it is also projective.

\begin{prop}
If $c\neq 0$ or $d\neq 0$, then $X(a,b,c,d)$ is an $N$-polarized K3 surface.
\end{prop}

\begin{proof} The condition $c\neq 0$ or $d\neq 0$ implies that the singular locus of $Y(a,b,c,d)$ consists of a finite number of rational double points. Since rational double points admit crepant resolutions, 
and $K_{Y(a,b,c,d)}$ is trivial, $X(a,b,c,d)$ is a K3 surface.

Now we want to construct the N-polarization. First observe that for any choice of parameters $(a,b,c,d)$, the variety $Y(a,b,c,d)$ always contains the singular points $$O=[0:1:0:0] \text{ and } P=[0:0:1:0].$$ 

To understand structure of the exceptional divisors obtained by blowing-up the points $O$ and $P$, we study first their singularity-type.

$O$ is always an $A_{11}$ singularity. To see this, using a computer system, we can check that $$\mu(Q_{a,b,c,d}(x,1,z,w))=11, rk(Hess(Q_{a,b,c,d}(x,1,z,w))(0,0,0))=2.$$ 

For $P$, its singularity-type depends on $c$. If $c=0$, it is an $E_6$ singularity since $\mu(Q_{a,b,c,d}(x,y,1,w))=6$ and $rank(Hess(Q_{a,b,c,d}(x,y,1,w))(0,0,0))=1$. On the other hand, $c\neq 0$ implies that $P$ is an $A_5$ singularity since $\mu(Q_{a,b,c,d}(x,y,1,w))=5$ and $rank(Hess(Q_{a,b,c,d}(x,y,1,w))(0,0,0))=2$ 

To define the $N$-polarization, we will need the lines $L_1''$ and $L_2''$, which are the proper transforms of the lines $x=w=0$ and $z=w=0$.

For $c\neq 0$, the intersection of the plane $x=\frac{d}{2c}w$ with $Y(a,b,c,d)$ has two components: the intersection of the plane $x=w=0$ with $Y(a,b,c,d)$ (which was already mentioned), and a rational curve $C$. 

Further details can be found in \cite[pp. 4-5]{dor}.
\end{proof}

Let us define the parameter space $$\mathsf{T}=\{(a,b,c,d)\in \mathbb{C}^4|\, c\neq 0 \text{ or } d\neq 0\}.$$

\begin{prop}
Let $(a,b,c,d)\in \mathsf{T}$. Then, for any $t\in \mathbb{C}^*$, the K3 surfaces

$$X(a,b,c,d) \text{ and } X(t^2a,t^3b,t^5c,t^5d)$$ 

are isomorphic as $N$-polarized algebraic varieties.
\end{prop}

%% file: s3.tex
In this section we will study the algebraic cohomology of a subfamily of the Clingher-Doran family of $N$-polarized K3 surfaces. This family is obtained as the minimal resolution of singularities of the family of singular complex 
projective varieties
\[Y_{a,b,c,d}:F_{a,b,c,d}=y^2zw-4x^3z+3axzw^2+bzw^3+cxz^2w-\frac{1}{2}(dz^2w^2+w^4)=0\]
for complex parameters $a,b,c,d,$ such that $c\neq 0$ or $d\neq 0$. Such projective varieties have only double rational singularities, therefore, its singular cohomology coincides with the hypercohomology of its Du Bois complex, which we recall in \S \ref{sdubois}.

We will be interested mainly in K3 surfaces of the Clingher-Doran family with the additional property that its singular locus consists of the points $[0,1,0,0]$ and $[0,0,1,0]$. The parameter space for such varieties is \[\mathsf{B}:=\{(a,b,c,d)\in\mathbb{C}^4|\, \mathcal{D}_4(a,b,c,d)\neq 0\},\]
where $\mathcal{D}_4$ is a complex polynomial which will be explicitely stated later \cite[Comments after Remark 2.3.]{dor}. This family will be denoted by $\pi:\mathcal{Y}\rightarrow\mathsf{B}$. By using the Poincaré residue sequence for orbifolds, we are able to use the work of \cite{humbert} to find five sections of the cohomology bundle $\mathcal{H}^2(\mathcal{Y}/\mathsf{B}):=\mathbb{R}^2\pi_*\tilde{\Omega}^{\bullet}_{\mathcal{Y}/\mathsf{B}}$, which are linearly independent on an open subset of $\mathsf{B}$. The pullback of these sections to the cohomology bundle $\mathcal{H}^2(\mathcal{X}/\mathsf{B}):=\mathbb{R}^2\pi_*\Omega^{\bullet}_{\mathcal{X}/\mathsf{B}}$, where $\mathcal{X}\rightarrow \mathsf{B}$ is the family obtained by fiberwise resolution of singularities of the family $\pi:\mathcal{Y}\rightarrow\mathsf{B}$, gives us linearly independent sections which do not belong to the image of the polarization. These sections will be used in the next section to construct a quasi-affine patch $\mathsf{O}$ of the moduli space $\mathsf{T}$ of enhanced $N$-polarized K3 surfaces.

Since the needed results of \cite{humbert} are in the context of \textit{tame polynomials} as developed in \cite{hodge1}, we explain the main parts of this theory, and make some comments about the relevant case of tame polynomials with null-dicriminant. 

\subsection{On $V$-varieties and their cohomologies}\label{sdubois}

The following theorem states the generalization of the de Rham complex to singular varieties done by Du Bois.

\begin{teo}
Let $Y$ be a complex scheme of finite type. Then, there exists a unique $\underline{\Omega}^{\bullet}_Y\in Ob(D_{filt}(Y))$, called the \textbf{Du Bois complex}, such that: 

\begin{enumerate}
\item $\underline{\Omega}_Y^{\bullet}\cong_{qis} \mathbb{C}_Y$, i.e., $\underline{\Omega}^{\bullet}_Y$ is a resolution of the constant sheaf $\mathbb{C}_Y$ on $Y$;
\item If $Y$ is proper, then $\underline{\Omega}^{\bullet}_Y\in Ob(D_{filt,coh}^b(Y))$, and there is a spectral sequence degenerating at $E_1$ and abutting to the singular cohomology of $Y$ such that the resulting filtration coincides with Deligne's Hodge filtration:

$$ E_1^{pq}=H^q(Y,\underline{\Omega}_Y^p)\Rightarrow H^{p+q}(Y,\mathbb{C}).$$
In particular, $$Gr^p_FH^{p+q}(Y,\mathbb{C})\cong H^q(Y,\underline{\Omega}_Y^p);$$

\item If $\rho:X\rightarrow Y$ is a resolution of singularities, then $\underline{\Omega}_Y^{\text{dim}Y}\cong_{qis}R\rho_*\omega_X$;
\end{enumerate}
\end{teo}

In our case, $Y_{a,b,c,d}$ is a $V$-variety in the following sense.

\begin{df}
A separated, complex scheme of finite type $Y$ is called a $V$-variety if, locally for the analytic topology, $Y$ is isomorphic to a quotient of a smooth  complex scheme $X$ by a finite group of automorphisms of $X$.
\end{df}

For a $V$-variety, the Du Bois complex has an explicit description which is due to Steenbrink. 

\begin{teo}\label{steen}
Let $Y$ be a $V$-variety, and $\Sigma$ its singular locus. Let $j:Y-\Sigma\rightarrow Y$ be the inclusion, and $\Omega_{Y-\Sigma}^{\bullet}$ be the de Rham complex of $Y-\Sigma$. Then, $\tilde{\Omega}_Y^{\bullet}:=j_{*}\Omega_{Y-\Sigma}^{\bullet}$ is quasi-isomorphic to the Du Bois complex $\underline{\Omega}_{Y}^{\bullet}$ of $Y$.
\end{teo}
\begin{proof}
\cite{steenbrink}.
\end{proof}

From now on, let $Y:=Y_{a,b,c,d}$ with $(a,b,c,d)\in \mathsf{B}$. and let $X:=X_{a,b,c,d}\rightarrow Y_{a,b,c,d}$ be its minimal resolution of singularities. They define a morphism $\rho:\mathsf{X}\rightarrow \mathsf{Y}$ of $\mathsf{B}$-schemes.

\begin{prop}
The singular locus of $Y$ is $\Sigma=\{[0,1,0,0],[0,0,1,0]\}$ if, and only if, $\mathcal{D}_4(a,b,c,d)\neq 0$. In such a case, both singularities are double rational singularities, and, in particular, $Y$ is an orbifold.

Here,
\begin{align*}
\mathcal{D}_{4}(a, b, c, d)=&-2^{5} 3^{6} a^{6} b c^{3}+2^{6} 3^{6} a^{3} b^{3} c^{3}-2^{5} 3^{6} b^{5} c^{3}-2^{4} 3^{5} a^{5} c^{4}+2^{4} 3^{5} 5^{2} a^{2} b^{2} c^{4}\\
&+2 \cdot 3^{3} 5^{4} a b c^{5}+
+5^{5} c^{6}-2^{4} 3^{7} a^{7} c^{2} d+2^{5} 3^{7} a^{4} b^{2} c^{2} d-2^{4} 3^{7} a b^{4} c^{2} d\\
&+2^{3} 3^{5} 5 \cdot 19 a^{3} b c^{3} d+2^{3} 3^{5} 5^{2} b^{3} c^{3} d+3^{3} 5^{3} 11 a^{2} c^{4} d+2^{3} 3^{5} 37 a^{4} c^{2} d^{2}\\
&+2^{3} 3^{5} 5 \cdot 7 a b^{2} c^{2} d^{2}-2^{3} 3^{3} 5^{3} b c^{3} d^{2}+2^{4} 3^{6} a^{6} d^{3}-2^{5} 3^{6} a^{3} b^{2} d^{3}\\
&+2^{4} 3^{6} b^{4} d^{3}-2^{6} 3^{6} a^{2} b c d^{3}-2^{3} 3^{5} 5^{2} a c^{2} d^{3}
-2^{5} 3^{6} a^{3} d^{4}-2^{5} 3^{6} b^{2} d^{4}\\
&+2^{4} 3^{6} d^{5}.
\end{align*}
\end{prop}
\begin{proof}
\cite[Theorem 2.2. and comments after Remark 2.3.]{dor}.
\end{proof}
Let us consider the $V$-divisor:\[Z=\{[x:y:z:w]\in Y|\, z=0\}.\]

The affine variety $U:=Y\backslash Z$ is an affine $V$-variety. Since we are working with $V$-varieties and $V$-divisors, we have a Poincaré residue sequence
\[0\rightarrow \tilde{\Omega}_Y^{\bullet}\rightarrow \tilde{\Omega}_Y^{\bullet}(\text{log}\, Z)\xrightarrow{\text{res}} j_{*}\tilde{\Omega}_Z^{\bullet -1}\rightarrow 0.
\]
It induces a long exact sequence in hypercohomology

\begin{equation}\label{poinca}\cdots\rightarrow\mathbb{H}^2(Y,\tilde{\Omega}_Y^{\bullet})\rightarrow \mathbb{H}^2(Y,\tilde{\Omega}_Y^{\bullet}(\text{log}\, Z))\xrightarrow{\text{res}} \mathbb{H}^2(Y,j_{*}\tilde{\Omega}_Z^{\bullet-1})\rightarrow \cdots.\end{equation}

Sequence \eqref{poinca} gives rise to the isomorphic long exact sequence

\begin{equation}\label{residue}\cdots\rightarrow H^2(Y,\mathbb{C})\rightarrow H^2(U,\mathbb{C})\xrightarrow{\text{res}} H^1(Z,\mathbb{C})\rightarrow \cdots.\end{equation}

Therefore, elements from $H^2(U,\mathbb{C})$ without residue are induced by elements in $H^2(Y,\mathbb{C})$. By the following theorem (Theorem \ref{deligne}), we can map monomorphically these elements into $H^2(X,\mathbb{C})$. In this way we will produce a basis of $H^2(X,\mathbb{C})_{\iota}$. To check that these elements do not belong to $\iota(N)$ we will verify that the Gauss-Manin connection does not vanish at them.

\begin{teo}\label{deligne}
If $Y$ is a complete complex $V$-variety, then the canonical Hodge structure on $H^k(Y,\mathbb{C})$ is pure of weight $k$, for all $k\geq 0$. If $\rho:X\rightarrow Y$ is a resolution of singularities for $X$, then $\rho^*:H^k(Y,\mathbb{C})\rightarrow H^k(X,\mathbb{C})$ is injective.
\end{teo}
\begin{proof}
\cite[Théorème 8.2.4.]{deligne}.
\end{proof}

\subsection{Cohomology of a deformation}\label{defor}

The following theorem will allow us to compute the cohomology $H^2(U,\mathbb{C})$ of the affine singular variety $U$ by considering a smooth deformation.

\begin{teo} \label{verdier}
Let $X$ and $Y$ be complex varieties, and $\pi:X\rightarrow Y$ a morphism. There is a Zariski open subset $U$ of $Y$, dense in $Y$, such that $\pi|_U:U\rightarrow Y$ is a locally trivial topological fibration for the analytic topology.
\end{teo}
\begin{proof}

\cite[Corollaire 5.1]{verdier}.
\end{proof}

For fixed parameters $(a,b,c,d)\in\mathsf{B}$, let us define the polynomial
$$ f_{a,b,c,d}:=F_{a,b,c,d}|_{z=1}=y^2w-4x^3+3axw^2+bw^3+cxw-\frac{1}{2}(dw^2+w^4).$$

Sometimes we will just write $f$ instead of $f_{a,b,c,d}$. Now, let us consider the closed subscheme $\mathcal{U}=Spec(\frac{\mathbb{C}[x,y,z,s]}{\langle f-s\rangle})$ of $\mathbb{A}^4_{\mathbb{C}}$. We have a natural morphism $\pi: \mathcal{U}\rightarrow \mathbb{A}_{C}$ induced by the inclusion $\mathbb{C}[s]\subset\mathbb{C}[x,y,z,s]$. By Theorem \ref{verdier} and its proof we conclude that $\pi: \mathcal{U}\rightarrow \mathbb{A}_{C}$ is a locally trivial topological fibration. By looking at the proof of Theorem \ref{verdier}, we see that we can choose $U$ in that proof to be the entire $\mathbb{A}_{\mathbb{C}}$ or at least an open ball around $0$. Let us write $\mathcal{U}=\{U_s\}_{s\in \mathbb{A}_{\mathbb{C}}}$. Observe that $U_0$ is equal to the affine variety $U$ in \S3.1. Therefore, for a small complex parameter $s$, we have the isomorphism
\begin{equation}\label{god}
H^2(U,\mathbb{C})\cong H^2(U_s,\mathbb{C}).
\end{equation}

Since $U_s$ is smooth except for a finite number of values of $s$, we can assume from now on that $s$ is a complex parameter such that $U_s$ is smooth and isomorphism \eqref{god} holds. The cohomology module $H^2(U_s,\mathbb{C})$ can be studied by means of the theory of \textit{tame polynomials}, which is explained in the next sections.

\subsection{Tame polynomials}

\textbf{Milnor, Tjurina modules and the discriminant}. Let $R$ be a commutative ring with unity. To any $f\in R[x_1,\ldots,x_{n+1}]$ we can associate the following $R$-algebras inspired from singularity theory: 

$$\mathsf{Milnor}(f):=\frac{R[x_1,\ldots,x_{n+1}]}{\mathsf{Jacob}(f)},$$
$$\mathsf{Tjurina}(f):=\frac{R[x_1,\ldots,x_{n+1}]}{\mathsf{Jacob}(f)+\langle f\rangle}.$$

The dimensions of the previous modules are denoted by $\mu(f)$ and $\tau(f)$, respectively, and are called the Milnor and Tjurina numbers of $f$. 

Tame polynomials were introduced in \cite[Chapters 7 and 10]{hodge1}. They are algebraic deformations of quasi-homogeneous polynomials with finite Milnor number.

\begin{df}
We say that $f$ is tame if there is a grading of the ring $R[x_1,\ldots,x_{n+1}]$ such that, if $g$ is the highest homogeneus degree part of $f$, then $\mathsf{Milnor(g)}$ is a free $R$-module of finite rank.
\end{df}

\textit{For this section, we fix a tame polymial $f$, and a grading of $R[x_1,\ldots,x_{n+1}]$ as in the previous definition. Let $d=deg(f)$ and $\nu_i=deg(x_i)$.}

We see $\mathbb{N}^{n+1}$ as a set of multi-indices. For any multi-index  $\beta=(\beta_1,\ldots,\beta_{n+1}),$ we define $x^{\beta}=x_1^{\beta_1}\cdots x_{n+1}^{\beta_{n+1}}.$ 
\begin{prop}\label{basis}
If $I\subset \mathbb{N}^{n+1}$ is such that $\mathcal{B}=\{x^{\beta}|\, \beta\in I\}$ is an $R$-basis for $\mathsf{Milnor}(g)$, then $\mathcal{B}$ is also an $R$-basis for $\mathsf{Milnor}(f)$.
\end{prop}
\begin{proof}
\cite[Proposition 10.7, page 143]{hodge1}.
\end{proof}

Information about the regularity of the variety $V(f)=Spec(\frac{R[x_1,\ldots,x_{n+1}]}{\langle f \rangle})$ is captured by the discriminant, which we define next.

\begin{df}
Let $T_f:\mathsf{Milnor}(f)\rightarrow\mathsf{Milnor}(f)$ given by $T_f(P)=fP$. Then, the discriminant of $f$ is defined to be $\Delta_f:=det(T_f)$.
\end{df}

\begin{prop}
If $\Delta_f\neq 0$, then $V(f)$ is regular. Furthermore, if $R$ is an algebraically closed field, the converse holds.
\end{prop}

\begin{proof}
\cite[Proposition 10.8, page 145]{hodge1}.
\end{proof}

The following property of the discriminant will be useful for the computation of the Gauss-Manin connection.

\begin{prop}\label{discriminant}
$\Delta_f\in \mathsf{Jacob}(f)+\langle f\rangle$. Equivalently, $\Delta_f\cdot\mathsf{Tjurina}(f)=0$.
\end{prop}

\begin{proof}
Let $p(s)=det(T_f-sI)\in R[s]$. By the Cayley-Hamilton theorem, $p(T_f)=0$. This is equivalent to $p(f)\in \mathsf{Jacob}(f)$. Therefore, $\Delta_f=p(0)\in \mathsf{Jacob}(f)+\langle f\rangle$.
\end{proof}

\begin{cor}
Suppose $R$ is a field. Then, $\Delta_f\neq 0$ if, and only if, $\tau(f)=0$.
\end{cor}

\noindent\textbf{Algebraic de Rham cohomology and Brieskorn modules.} Let us introduce the notation  $\mathbb{A}^n_R:=Spec(R[x_1,\ldots,x_{n+1}]), \mathsf{T}:=Spec(R), dx:=dx_1\wedge\cdots \wedge dx_{n+1},$ $$\widehat{dx_i}:=dx_1\wedge\cdots dx_{i-1}\wedge dx_{i+1}\wedge \cdots \wedge dx_{n+1},$$ $$ \eta:=\sum_{i=1}^{n+1}(-1)^{i+1}\frac{\nu_i}{d}x_i\widehat{dx_i},\hspace{5mm} \nu_i=deg(x_i).$$

\begin{prop}\label{eisenbud} For every $k\geq0$, we have
\begin{equation*}
\Omega^k_{V(f)/\mathsf{T}}\cong\frac{\Omega^k_{\mathbb{A}^n_R/\mathsf{T}}}{f\Omega^k_{\mathbb{A}^n_R/\mathsf{T}}+df\wedge \Omega^{k-1}_{\mathbb{A}^n_R/\mathsf{T}}}.
\end{equation*}
\end{prop}
\begin{proof}
\cite[Example 1.10, page 213]{Liu}.
\end{proof}

\begin{cor}
$\mathsf{Tjurina}(f)\cong\Omega^{n+1}_{V(f)/\mathsf{T}}$. Therefore, $\Delta_f\cdot\Omega^{n+1}_{V(f)/\mathsf{T}}=0$. 
\end{cor}

\begin{proof}
The first statement is given by the isomorphism $P\mapsto Pdx$. The second statement follows from the first one and Proposition \ref{discriminant}.
\end{proof}

\begin{cor} $H_{dR}^n(V(f)/\mathsf{T})_{\Delta_f}\cong (\frac{\Omega^n_{\mathsf{A}/\mathsf{T}}}{f\Omega^n_{\mathsf{A}/\mathsf{T}}+df\wedge \Omega^{n-1}_{\mathsf{A}/\mathsf{T}}+d\Omega^{n-1}_{\mathsf{A}/\mathsf{T}}})_{\Delta_f}$.
\end{cor}
\begin{proof}
$H_{dR}^n(V(f)/\mathsf{T})_{\Delta_f}=coker((\Omega^{n-1}_{V(f)/\mathsf{T}})_{\Delta_f}\xrightarrow{d}(\Omega^{n}_{V(f)/\mathsf{T}})_{\Delta_f})$, since, by the previous corollary,  $(\Omega^{n+1}_{V(f)/\mathsf{T}})_{\Delta_f}=0$. The result now follows from Proposition \ref{eisenbud}.
\end{proof}
This proposition motivates the following one.

\begin{df} the Brieskorn modules of $f$ are the $R$-modules $$\mathsf{H}_f=\frac{\Omega^{n+1}_{\mathbb{A}^n_R/\mathsf{T}}}{f\Omega^{n+1}_{\mathbb{A}^n_R/\mathsf{T}}+df\wedge d\Omega^{n-1}_{\mathbb{A}^n_R/\mathsf{T}}},$$
$$\mathsf{H}'_f=\frac{\Omega^n_{\mathbb{A}^n_R/\mathsf{T}}}{f\Omega^n_{\mathbb{A}^n_R/\mathsf{T}}+df\wedge \Omega^{n-1}_{\mathbb{A}^n_R/\mathsf{T}}+d\Omega^{n-1}_{\mathbb{A}^n_R/\mathsf{T}}}.$$
\end{df}

\begin{prop} 
$(\mathsf{H}'_f)_{\Delta_f}\cong H_{dR}^n(V(f)/\mathsf{T})_{\Delta_f}$. If $R$ is Cohen-Macaulay and $\Delta_f\neq 0$, then $\mathsf{H}'_f\cong \Delta_f\cdot \mathsf{H}_f$.
\end{prop}
\begin{proof}
The first assertion follows from the last corollary. For the second assertion, consider the homomorphism $\iota:\mathsf{H}'_f\rightarrow \mathsf{H}_f,\omega\mapsto df\wedge \omega$. It is injective by \cite[Proposition 10.9]{hodge1}. On the other hand, $\mathsf{Tjurina}(f)$ is isomorphic to the quotient $\mathsf{H}_f/im(\iota)$ by means of the map $P\mapsto Pdx$. From Proposition \ref{discriminant}, we conclude that $im(\iota)=\Delta_f\cdot\mathsf{H}_f$.
\end{proof}
\begin{cor}[From the proof]\label{emb} Suppose $R$ is Cohen-Macaulay and $\Delta_f\neq 0$. Let $\iota:\mathsf{H}'_f\rightarrow \mathsf{H}_f,\omega\mapsto df\wedge \omega$. Then, $\iota$ is a monomorphism such that $coker(\iota)\cong \mathsf{Tjurina}(f)$. 

\end{cor}
\begin{cor}
If $R$ is Cohen-Macaulay and $\Delta_f\neq 0$, then $H_{dR}^n(V(f)/\mathsf{T})_{\Delta_f}\cong (\mathsf{H}_f)_{\Delta_f}\cong (\mathsf{H}'_f)_{\Delta_f}$.
\end{cor}

The following theorem gives us explicit bases for the Brieskorn modules $\mathsf{H}_f$ and $\mathsf{H}_f'$ in terms of a basis for $\mathsf{Milnor}(f)$. Together with the previous corollary, it allows us to compute $H^n_{dR} (V(f)/\mathsf{T})$ in many cases. We use the notations of Proposition \ref{basis}.

\begin{teo}\label{ases}
Suppose $R$ is Cohen-Macaulay of characteristic zero and $\mathbb{Q}\subset R$. Then, the sets $\{x^{\beta}dx|\, \beta\in I\}$ and $\{x^{\beta}\eta|\, \beta\in I\}$ are $R$-bases for $\mathsf{H}_f$ and $\mathsf{H}_f'$, respectively.
\end{teo}
\begin{proof}\cite[Corollary 10.1, page 150]{hodge1}.
\end{proof}

The bases in Theorem \ref{ases} will be extensively used in this section, so the following notation will be convenient:

\begin{equation}
\omega_{\beta}:=x^{\beta}dx,\hspace{5mm} \eta_{\beta}:=x^{\beta}\eta, \,\, \beta\in I.
\end{equation}

\noindent\textbf{Gauss-Manin connection and Gauss-Manin system.} \textit{During this section, we suppose that $\Delta_f\neq 0$, and that $R$ is of the form $\mathbb{Q}[a_1,\ldots,a_l]$, for some parameters $a_1,\ldots,a_l$.} 

\begin{df} The Gauss-Manin system associated to $f$ is the $R$-module 
\begin{equation*}
\mathsf{M}_f=\frac{\Omega_{\mathsf{A}/\mathsf{T}}^{n+1}[\frac{1}{f}]}{\Omega_{\mathsf{A}/\mathsf{T}}^{n+1}+d(\Omega_{\mathsf{A}/\mathsf{T}}^{n}[\frac{1}{f}])}.
\end{equation*}

The pole filtration of $\mathsf{M}_{\bullet}$ of $\mathsf{M}_f$ by $R$-modules is
\begin{equation*}
\mathsf{M}_l:=\{\frac{\omega}{f^l}\in\mathsf{M}_f\}.
\end{equation*}
\end{df}

\begin{prop}\label{degreeone} The map $\omega\mapsto \frac{\omega}{f}$ defines an isomorphism $\mathsf{H}_f\cong \mathsf{M}_1$
\end{prop}
\begin{proof}
\cite[Proposition 11.2, page 159]{hodge1}.
\end{proof}

By the previous proposition, we can identify $\mathsf{H}_f$ with $\mathsf{M}_1$. By Corollary \ref{emb}, $\iota:\mathsf{H}'_f\rightarrow \mathsf{H}_f,\omega\mapsto df\wedge \omega$ is an injective homomorphism. Let us define $\mathsf{M}_0:= im(\iota)$. Then, the same corollary implies that $\mathsf{M}_1/\mathsf{M}_0\cong \mathsf{Tjurina}(f)\cong \Omega^{n+1}_{V(f)/\mathsf{T}}$. The following proposition extends this result to the whole filtration $\mathsf{M}_{\bullet}$.

\begin{prop}
For each $l\geq 1$, the map $\omega\mapsto \frac{\omega}{f^l}$ defines an isomorphism $\Omega^{n+1}_{V(f)/\mathsf{T}}\cong \mathsf{M}_l/\mathsf{M}_{l-1}$. In particular, $(\mathsf{M}_l)_{\Delta_f}=(\mathsf{M}_{l-1})_{\Delta_f}$. 
\end{prop}
\begin{proof}
\cite[Proposition 11.2, page 159]{hodge1}.
\end{proof}

The two previous propositions allow us to conclude that $$(\mathsf{M}_{f})_{\Delta_f}=(\mathsf{M}_1)_{\Delta_f}\cong (\mathsf{H}_f)_{\Delta_f}\cong (\mathsf{H}'_f)_{\Delta_f}.$$ 

The main importance of the introduction of the Gauss-Manin system is that it allows us to explicitly define important concepts like the Gauss-Manin connection and the mixed Hodge structure associated to these modules.

The Gauss-Manin connection can be directly defined on $\mathsf{M}$, as done in \cite{hodge1}. Let
\begin{equation*}
d_{par}:R[x_1,\ldots,x_{n+1}]\rightarrow \Omega_{\mathsf{T}}\otimes_{R}R[x_1,\ldots,x_{n+1}]
\end{equation*} be the differentiation with respect to parameters $a_1,\ldots,a_l$. 

\begin{df}The \textbf{Gauss-Manin connection} on $\mathsf{M}$ is defined to be 
$$\nabla:\mathsf{M}\rightarrow \Omega_{\mathsf{T}}\otimes_{R} \mathsf{M},$$
$$\nabla(\frac{Pdx}{f^k})=\frac{(d_{par}P)f-kP(d_{par}f)}{f^{k+1}}dx.$$ 
\end{df}
Observe that applying $\nabla$ apparently increases the pole order. If we allow division by the discriminant $\Delta$, we can indeed reduce the pole order. We proceed to explain this in detail.
\begin{prop}
\begin{equation*}
\nabla(\mathsf{M}_i)\subset \frac{1}{\Delta}\Omega_{\mathsf{T}}\otimes_{R} \mathsf{M}_i.
\end{equation*}
\end{prop}

\begin{proof}
Let $\omega=Pdx\in \Omega_{\mathsf{A}/\mathsf{T}}^{n+1}, P\in R[x_1,\ldots,x_{n+1}].$ Then, by Proposition \ref{discriminant}, $\Delta P=fQ+\sum_i Q_i\frac{\partial f}{\partial x_i},$ for some $Q,Q_i \in R[x_1,\ldots,x_{n+1}].$ Therefore, $\Delta \omega=f\omega_1+df\wedge d\omega_2$, with $\omega_1=Qdx\in \Omega_{\mathsf{A}/\mathsf{T}}^{n+1}.$ In this way we conclude $\frac{\omega}{f^k}=\frac{1}{\Delta}\frac{\omega_1}{f^{k-1}}$ in $\mathsf{M}_{\Delta.}$
\end{proof}
Using this reduction procedure, we can reduce poles until degree at most $1$, which gives us the Gauss-Manin connection \begin{equation*}
\nabla:\mathsf{H}\rightarrow \frac{1}{\Delta}\Omega_{\mathsf{T}}\otimes_R \mathsf{H}.
\end{equation*}

For the following proposition, let us define $A_{\beta}:=\frac{deg(\omega_{\beta})}{d}$ for every $\beta\in I$. 

\begin{df}
We define $W_n=W_n(\mathsf{M}_f)_{\Delta_f}$ to be the $R_{\Delta_f}$-submodule of $(\mathsf{M}_f)_{\Delta_f}$ generated by the forms $$\frac{\omega_{\beta}}{f^l},\, \beta\in I,\, A_{\beta}<l,$$

and call $$0:=W_{n-1}\subset W_n\subset W_{n+1}=:(\mathsf{M}_f)_{\Delta_f}$$ the \textbf{weight filtration} of $(\mathsf{M}_f)_{\Delta_f}$.

We also define $F^i=F^i(\mathsf{M}_f)_{\Delta_f}$ to be the $R_{\Delta_f}$-submodule of $(\mathsf{M}_f)_{\Delta_f}$ generated by the forms $$\frac{\omega_{\beta}}{f^l},\, \beta\in I,\, A_{\beta}\leq l\leq n+1-i,$$

and call $$0=F^{n+1}\subset F^n\subset F^{n-1}\subset\cdots\subset F^0$$ the \textbf{Hodge filtration} of $(\mathsf{M}_f)_{\Delta_f}$.

\end{df}

\noindent\textbf{Tame polynomials applied to our case.} Since our polynomial $$ f=y^2w-4x^3+3axw^2+bw^3+cxw-\frac{1}{2}(dw^2+w^4)$$

has null-discriminant, we cannot apply directly tame polynomial theory. Let us consider a complex number $s$ as in \S \ref{defor}. Since $U_s$ is smooth, we have that $f-s$ is tame polynomial with non-zero discriminant. Therefore, the Brieskorn modules $\mathsf{H}_{f-s}$, $\mathsf{H}'_{f-s}$, and the Gauss-Manin system $\mathsf{M}_{f-s}$ are isomorphic to $H^2(U_s,\mathbb{C})$.

Let us consider the form

$$ \omega_1:=\omega_{(0,0,0)}=\frac{dx\wedge dy \wedge dw}{f-s}.$$

Then, $\omega_1\in F^2\mathsf{M}_{f-s}$. Since $A_{(0,0,0)}=\frac{23}{24}<1$, $\omega_1$ has no poles at infinity. Therefore, this form under the isomorphism \eqref{god} gives us a form $\omega\in H^2(U,\mathbb{C})$ with zero residue. By the Poincaré residue sequence, this form comes from a form, which we will also call $\omega$, such that $\omega\in F^2H^2(Y,\mathbb{C})$ (recall that by Theorem \ref{deligne}, $H^2(Y,\mathbb{C})$ has a pure Hodge structure of weight $2$). Therefore, this form is mapped to a holomorphic for under the pullback $\rho^*:H^2(Y,\mathbb{C})\rightarrow H^2(X,\mathbb{C})$ of the resolution of singularities $\rho:X\rightarrow Y$. 

By applying the Gauss-Manin connection defined in the previous section to $\omega_1$ to produce $\omega_2,\ldots,\omega_5$ such that $\omega_2,\ldots,\omega_4\in F^1\mathsf{M}_{f-s}$, and $\omega_1,\ldots,\omega_5$ are linearly independent over $\mathbb{Q}(a,b,c,d,s)$, therefore linearly independent over a codimension one Zariski open subset of $Spec(\mathbb{Q}[a,b,c,d,s])$. One way of doing this was pursued in \S6.9 of \cite{humbert}. Another path for finding such a linearly independent set of forms would involve the following theorem

\begin{teo}\label{compatible}
For the integers $d_{i,j,k}$ in the table below, the sets

$$\mathcal{B}^m_2:=\{\frac{\omega_{i,j,k}}{\tilde{f}^{3-m}}|\, 3-m-\frac{d_{i,j,k}+1}{24}<A_{i,j,k}<3-m\},$$
$$\mathcal{B}^m_3:=\{\frac{\omega_{i,j,k}}{\tilde{f}^{3-m}}|\,A_{i,j,k}=3-m\}$$

form a basis of the $\mathbb{C}$-modules $Gr_F^m Gr_2^W\mathsf{H}$ and $Gr_F^m Gr_3^W\mathsf{H}$ specialized at parameters $(a,b,c,d,s)$ that lie outside a certain algebraic $\Sigma$ locus of codimension $1$, which includes the discriminant locus.

\begin{center}
\begin{tabular}{|c|c|} 
\hline
$(i,j,k)$ & $d_{i,j,k}$ \\
\hline
$(1,0,3)$& $3$\\
$(1,0,2)$& $3$\\
$(0,0,3)$& $13$\\
$(1,1,0)$& $23$\\
$(1,0,1)$& $21$\\
$(0,0,2)$& $25$\\
$(1,0,0)$& $23$\\
$(0,1,0)$& $33$\\
$(0,0,1)$& $37$\\
$(0,0,0)$& $49$\\

\hline
\end{tabular}
\end{center}

The vanishing locus is $$\Sigma=V(\Delta_{f})\cup V(c_{1,0,3}) \cup V(c_{1,0,2}) \cup V(b_{1,0,1})\cup V(c_{1,0,1}) \cup V(c_{0,0,2})\cup V(s) \cup V(a) \subset \mathbb{C}^5$$

\end{teo}

\begin{proof}
This is an implementation of the algorithm described in the proof of \cite[Proposition 11.8.]{hodge1}
\end{proof}

For the description of the locus $\Sigma$, see the Appendix.

%% file: s4.tex
To begin studying systematically a certain class of objects in mathematics, it is necessary to define the interested structure those objects share, and consider the functions between objects that preserve the structure, usually called morphisms. Isomorphisms can be defined (at least in this setting) as bijective morphisms whose inverse is also a morphism. They induce an equivalence relation between the objects under consideration. Isomorphic objects are considered as essentially the same objects up to the structure under consideration.

In \textit{moduli theory} we are given such a class of objects, and we consider the set $\mathcal{M}$ of its isomorphy classes, which we call a \textit{moduli space}. Sometimes $\mathcal{M}$ itself comes equipped with some structure (e.g., topological, differential or algebraic). The study of the properties of $\mathcal{M}$ usually enlightens the properties of the original class of objects under consideration. In some cases, the space $\mathcal{M}$ is interesting in itself.

In our setting, the objects under consideration will always be algebraic varieties with some additional structure given by cohomological data, and the morphisms will be morphisms between the underlying varieties whose pullbacks in cohomology preserve the additional structure. 

In this section we study two moduli spaces $\mathsf{S}$ and $\mathsf{T}$, introduced by Movasati in the framework of elliptic curves \cite{pascal} and mirror quintics \cite{mov17}, whose rings of regular functions under a pullback by the mirror map defined in the next section will give us meromorphic Siegel modular forms of genus 2 in the former case, and a further generalization which we call meromorphic Siegel quasi-modular forms, in the latter. In each case we have algebraic groups acting on the moduli spaces which encode the automorphic properties of the modular forms.

\subsection{The moduli space $\mathsf{S}$}

Recall that the family $\mathcal{X}
\rightarrow \mathsf{B}$ of $N$-polarized $K3$ surfaces introduced in \cite{dor} was obtained by considering the hypersurfaces $Y_{a,b,c,d}$ in $\mathbb{P}^3$ given by the zero locus of the polynomial 
\begin{equation*}
F_{a,b,c,d}=y^2zw-4x^3z+3axzw^2+bzw^3+cxz^2w
-\frac{1}{2}(dz^2w^2+w^4),
\end{equation*}

with parameters living in the Zariski open  subset $\mathsf{B}=\mathbb{C}^{4}-\{c=d=0\},$ and then by taking its minimal resolutions of singularities $X_{a,b,c,d}$.

Let us explain the origin of the $N$-polarization. The rational projection $$\pi:\mathbb{P}^3\dashrightarrow \mathbb{P}^1,$$
$$[x,y,z,w]\mapsto[z,w],$$

induces a rational map $X_{a,b,c,d}\dashrightarrow \mathbb{P}^1$. By Proposition \ref{extension}, it extends to a regular map $X_{a,b,c,d}\rightarrow \mathbb{P}^1$ since K3 surfaces have trivial canonical class.

\begin{prop}\label{extension}
Every rational function on a K3 surface extends to a regular function.
\end{prop}
\begin{proof}
Let $X$ be a K3 surface, and $f$ a rational function defined on $X$. Since $X$ has trivial canonical class, $\text{div}(f)$ is linearly equivalent to the zero divisor. Then, $\text{div}(f)$ or $-\text{div}(f)$ is effective. The proof follows from observing that $f$ vanishes somewhere in $X$.
\end{proof}

 This is an elliptic fibration, which we denote by $$\varphi_{a,b,c,d}^s:X_{a,b,c,d}\rightarrow \mathbb{P}^1,$$ and is called the \textbf{standard fibration} in \cite{dor}.

\begin{teo}[\cite{dor}]\label{fibers} For, $c\neq 0$ or $d\neq 0$, the standard fibration $\varphi_{a,b,c,d}^s:X_{a,b,c,d}\rightarrow \mathbb{P}^1$ has a section, and two special singular fibers over the base points $[0,1]$ and $[1,0].$ The fiber over $[0,1]$ is of type Kodaira type $II^*$.
The fiber over $[1,0]$ is of type $III^*$ if $c\neq0$, and of type $II^*$ if $c=0$.
\end{teo}

Since the resolution of singular fibers of type $II^*$ and $III^*$ gives rise to divisors with configurations $E_8$ and $E_7$, respectively, the $N$-polarization of $X_{a,b,c,d}$ can be obtained from the section, the general fiber and the two special fibers.

\begin{df}
A $N$-polarized K3 surface whose polarization does not extend to a $M$-polarization will be called a \textbf{strictly $N$-polarized} K3 surface.
\end{df}

\begin{prop}
Every strictly $N$-polarized K3 surface is isomorphic to one of the form $X_{a,b,c,d}$, with $c\neq 0$.
\end{prop}

\begin{proof}
By Theorem 1.2 in \cite{dor}, every $N$-polarized K3 surface is isomorphic to one of the form $X_{a,b,c,d}$, with $(a,b,c,d)\in \mathsf{B}$. Theorem \ref{fibers} implies that $c\neq 0$. 
\end{proof}

Theorem \ref{fibers} characterizes the $N$-polarization in such a way that will be useful for us to characterize the automorphisms of strictly $N$-polarized K3 surfaces preserving the $N$-polarization.

\begin{prop} If $c\neq 0$, every automorphism
$\varphi:X_{a,b,c,d}\rightarrow X_{a,b,c,d
}$ preserving the $N$-polarization preserves the fibers of the standard fibration, i.e., $\varphi^s_{a,b,c,d}\circ\varphi=\varphi^s_{a,b,c,d}$.
\end{prop}

\begin{proof} Suppose $c\neq 0$. Let $\varphi:X_{a,b,c,d}\rightarrow X_{a,b,c,d
}$ be an automorphism of $N$-polarized K3 surfaces. Let us fix the notation $0:=[0,1], $ and $\infty:=[1,0]$. Let us also wite $X:=X_{a,b,c,d}$ and $\varphi^s:=\varphi^s_{a,b,c,d}$. By Theorem \ref{fibers}, $(\varphi^s)^{-1}(0)=E_8$ and $(\varphi^s)^{-1}(\infty)=E_7$. Since $\varphi$ is an automorphism of $X$ preserving the polarization, we have that the divisors $\varphi^{-1}(E_8)$ and $\varphi^{-1}(E_7)$ are linearly equivalent to $E_8$ and $E_7$, respectively. But, since every rational function on $X$ extends to a regular function, we must have that $\varphi^{-1}(E_8)=E_8$ and $\varphi^{-1}(E_7)=E_7$. Now, choose a point $P\in\mathbb{P}^1$ such that $F:=(\varphi^s)^{-1}(P)$ is a smooth generic fiber. Then, $\varphi(F)$ is also a smooth generic fiber of $\varphi^s$. Since, $\varphi(F)$ and $F$ are linearly equivalent, Proposition \ref{extension} implies that $\varphi(F)=F$. This implies that $\varphi$ preserves the fibers of the elliptic fibration $\varphi^s:X\rightarrow \mathbb{P}^1$.
\end{proof}

\begin{prop}\label{order}
Every automorphism $\varphi$ of the K3 surface $X_{a,b,c,d}$, with $c\neq 0$ which preserves the $N$-polarization is of the form $\varphi([x,y,z,w])=[\kappa x, \mu y,z, w]$, where $\kappa,\mu\in\{1,-1\}$. Therefore, these automorphisms form a group isomorphic to $\mathbb{Z}/2\mathbb{Z}\times \mathbb{Z}/2\mathbb{Z}$. In particular, every such automorphism has order two.
\end{prop}
\begin{proof}
From $\varphi^s_{a,b,c,d}\circ\varphi=\varphi^s_{a,b,c,d}$, and the definition of $\varphi^s_{a,b,c,d}$ we get that $$\varphi([x,y,z,w])=[f(x,y,z,w),g(x,y,z,w),z,w]$$ for some linear homogeneous polynomials $f$ and $g$. Since $\varphi$ preserves the $N$-polarization, we must have $\varphi([0,1,0,0])=[0,1,0,0]$ and $\varphi([0,0,1,0])=[0,0,1,0]$. This implies, since $\varphi$ is an isomorphism, that it is of the form $\varphi([x,y,z,w])=[\kappa x,\lambda y,z, w]$ for some $\kappa,\lambda\in \mathbb{C}^*$. Finally, from $\varphi^{-1}(V(F_{a,b,c,d}))=V(F_{a,b,c,d})$ and by looking at the coefficients of $F_{a,b,c,d}$ we get that $\kappa^2=\lambda^2=1$. This concludes the proof.

\end{proof}
Let us consider the action of the algebraic group $\mathbb{C}^*$ on $\mathsf{B}$ given by
\begin{equation}\label{action} \lambda\cdot(a,b,c,d)=(\lambda^2 a,\lambda^3 b,\lambda^5 c,\lambda^6 d).
\end{equation}
\begin{teo}[\cite{dor}]\label{coarse} A coarse moduli space of $N$-polarized K3 surfaces is obtained as the quotient $$\mathcal{M}=\mathsf{B}/\mathbb{G}_m.$$ The underlying analytic space can be regarded as a Zariski open subset of the weighted projective space $\mathbb{P}(2,3,5,6)$ given by the condition $c\neq 0$ or $d\neq 0$.
\end{teo}
\begin{df}\label{smoduli}
Les us denote by $\mathsf{S}$ the moduli of triples $(X,\iota,\omega)$ such that:

\begin{itemize}
\item[i.] $X$ is a smooth complex algebraic strictly $N$-polarized K3 surface;
\item[ii.] $\iota:N\rightarrow H_{dR}^2(X/\mathbb{C})$ is a lattice polarization;
\item[iii.] $\omega\in F^2H^2_{dR}(X/\mathbb{C})$ is non-zero.
\end{itemize}

\end{df}

Sometimes we will omit the polarization $\iota$, and write simply $(X,\omega)$.
\begin{lem}\label{isouseful}
Let $\lambda\in \mathbb{C}^{*}$, and $q$ a square-root of $\lambda$. Let $\phi:\mathbb{P}^3\rightarrow \mathbb{P}^3$ be the projective isomorphism given by $\phi
([x,y,z,w])=[q^8x,q^9y,z,q^6w].$ Then, $\phi$ sends $Y_{a,b,c,d}$ onto $Y_{\lambda^2 a,\lambda^3 b,\lambda^5 c,\lambda^6 d}$. This isomorphism lifts to an isomorphism $X_{a,b,c,d}\cong X_{\lambda^2 a,\lambda^3 b,\lambda^5 c,\lambda^6 d}$ of $N$-polarized K3 surfaces. Furthermore, $\phi^{*}(res(\frac{\Omega}{F_{\lambda^2 a,\lambda^3 b,\lambda^5 c,\lambda^6 d}}))=q^{-1}\cdot res(\frac{\Omega}{F_{a,b,c,d}})$.
\end{lem}
\begin{proof}
Except for the last one, all af the previous statements are proved in. For the last statement, recall that $\Omega$ is induced by $\iota_EdV,$ where $E$ and $dV$ are the Euler vector field and the volume form of $\mathbb{C}^4$, respectively. By an abuse of notation, let us write $\phi:\mathbb{C}^4\rightarrow \mathbb{C}^4,\phi
(x,y,z,w)=(q^8x,q^9y,z,q^6w)$. Then $\phi^*(\Omega)= \phi^*(\iota_EdV)=\iota_{\phi^*(E)}\phi^*(dV)=det(\phi)\iota_EdV=q^{23}\Omega.$ Since $\phi^*(F_{\lambda^2 a,\lambda^3 b,\lambda^5 c,\lambda^6 d})=q^{24}F_{a,b,c,d}$, we conclude that $\phi^{*}(res(\frac{\Omega}{F_{\lambda^2 a,\lambda^3 b,\lambda^5 c,\lambda^6 d}}))=res(\phi^{*}(\frac{\Omega}{F_{\lambda^2 a,\lambda^3 b,\lambda^5 c,\lambda^6 d}}))=q^{-1}\cdot res(\frac{\Omega}{F_{a,b,c,d}})$.
\end{proof}

\begin{teo}\label{s}$\mathsf{S}$ is isomorphic to $V:=\{(a,b,c,d)\in\mathbb{C}^4|\, c\neq 0\}.$ 
\end{teo}

\begin{proof}
Let us consider the morphism $\Psi:\mathsf{B}\rightarrow \mathsf{S}, 
\Psi(a,b,c,d)=(X_{a,b,c,d},res(\frac{\Omega}{F_{a,b,c,d}})).$
To prove surjectivity, let $(X,\omega)$ be as in Definition \ref{smoduli}. $X$ is isomorphic as an $N$-polarized K3 surface to one of the form $X_{a,b,c,d}$ for some $(a,b,c,d)\in \mathsf{B}$. Since $h^{2,0}=1$ for any K3 surface, we have $(X,\omega)\cong(X_{a,b,c,d},k\cdot res(\frac{\Omega}{F_{a,b,c,d}}))$ for some $k\in \mathbb{C}^*$. By making $q=k^{-1}$ and $\lambda=k^{-2}$ in the previous lemma, we get \begin{equation}\label{weights}(X_{a,b,c,d},k\cdot res(\frac{\Omega}{F_{a,b,c,d}}))\cong (X_{k^{-4}a,k^{-6}b,k^{-10}c,k^{-12}d},res(\frac{\Omega}{F_{k^{-4}a,k^{-6}b,k^{-10}c,k^{-12}d}})).\end{equation}.

Next, we deal with injectivity. Let $(X_{a,b,c,d},res(\frac{\Omega}{F_{a,b,c,d}}))$ and $(X_{a',b',c',d'},res(\frac{\Omega}{F_{a',b',c',d'}}))$ be isomorphic. In particular, $X_{a,b,c,d}$ is isomorphic to $X_{a',b',c',d'}$ as $N$-polarized K3 surfaces, and, by Theorem \ref{coarse}, there is $\lambda\in \mathbb{C}$ such that $$(a',b',c',d')=(\lambda^2 a,\lambda^3 b,\lambda^5 c,\lambda^6 d).$$ Let $q$ be a square root of $\lambda$. Then, Lemma \ref{isouseful} implies that $(X_{a,b,c,d},res(\frac{\Omega}{F_{a,b,c,d}}))$ is isomorphic to $(X_{a,b,c,d},q^{-1}res(\frac{\Omega}{F_{a,b,c,d}}))$. Since such an automorphism of $X_{a,b,c,d}$ has order at most two, we have $\lambda=q^2=1$. This concludes the proof.
\end{proof}

The moduli space $\mathcal{M}$ can be recovered from $\mathsf{S}$ by forgetting the cohomological information introduced in the definition of the latter. More precisely, we have:

\begin{df} The algebraic group $\mathbb{C}^*$ acts on $\mathsf{S}$ on the left by means of
\begin{equation}
(X,\omega)\cdot k=(X,k\cdot\omega),
\end{equation}
\end{df}
The main properties of the previous action are captured in the following proposition. 

\begin{prop}
 i. $\mathsf{S}/\mathbb{C}^*\cong \mathcal{M}$; ii. $(a,b,c,d)\cdot k=(k^{-4}a,k^{-6}b,k^{-10}c,k^{-12}d)$.
\end{prop}
\begin{proof}
The first statement is immediate after recalling the definitions of the moduli spaces $\mathcal{M}$ and $\mathsf{S}$. The second statement follows from Equation \eqref{weights}.
\end{proof}

\subsection{The moduli space $\mathsf{T}$}

Let

\begin{equation}\label{psi1}
\Psi:= \begin{bmatrix}
0 & 0 & 1 \\
0 & \Psi' & 0\\
1 & 0 & 0 
\end{bmatrix}
\end{equation}
be a matrix such that $\Psi'$ is a non-singular and symmetric $3\times 3$ matrix with complex entries. 

\begin{df}\label{tmoduli1}
Les us denote by $\mathsf{T}_{\Psi}$ the moduli of tuples $(X,\iota,\alpha_1,\ldots,\alpha_5)$ such that:
\begin{itemize}
\item[i.] $X$ is a smooth complex algebraic $N$-polarized K3 surface;
\item[ii.] $\iota:N\rightarrow H_{dR}^2(X/\mathbb{C})$ is a lattice polarization;
\item[iii.] $(\alpha_1,\ldots,\alpha_5)$ is a basis of $H^2_{dR}(X/\mathbb{C})_{\iota}:=H_{dR}^2(X/\mathbb{C})/\iota(N)$ such that $\alpha_1\in F^2$, $\alpha_1,\ldots,\alpha_4\in F^1$ and $[\langle\alpha_i,\alpha_j \rangle]=\Psi$. Here, $H_{dR}^2(X/\mathbb{C})$ denotes the second algebraic de Rham cohomology group of $X$ over $\mathbb{C}$, and $F^{\bullet}$ denotes its Hodge filtration.
\end{itemize}
\end{df}

We observe that for any choice of $\Psi$ the resulting moduli spaces $\mathsf{T}_{\Psi}$ are canonically isomorphic over $\mathbb{C}$. More explicitly, let \begin{equation*}
\tilde{\Psi}:= \begin{bmatrix}
0 & 0 & 1 \\
0 & \tilde{\Psi}' & 0\\
1 & 0 & 0 
\end{bmatrix}
\end{equation*} be another matrix, where $\tilde{\Psi}'$ is a non-singular and symmetric $3\times 3$ matrix with complex entries.
\begin{prop}\label{invariance}
$\mathsf{T}_{\Psi'}$ and $\mathsf{T}_{\tilde{\Psi}'}$ are homeomorphic.
\end{prop}
\begin{proof}
 Since $\Psi'$ and $\tilde{\Psi}'$ are both non-singular and symmetric, they are congruent over $\mathbb{C}$. Then, there is a non-singular matrix $P$ such that $P^T\Psi'P=\tilde{\Psi}'$. Let us define a map $F:\mathsf{T}_{\Psi'}\rightarrow\mathsf{T}_{\tilde{\Psi}'}$ by $F(X,\iota,\alpha_1,\ldots,\alpha_5)=(x,\iota,\sum_j P_{1j}\alpha_j,\ldots,\sum_j P_{5j}\alpha_j)$. It is straightforward to check that it is well-defined on the isomorphy classes, and that it is continuous. It has a continuous inverse given by the same contruction applied to the matrix $P^{-1}$.  
\end{proof}

It is this freedom of choice that allows us to choose $\Psi$ of a convenient form and obtain results consistent with the choice of \cite[page 7]{alimtt}. From now on we may omit the intersection matrix $\Psi$ unless we make a particular choice of it.

\begin{rem}\label{notinvariant}On the other hand, the moduli space $\mathsf{T}$ might also be defined over $\mathbb{Q}$ or $\mathbb{Z}[\frac{1}{N}]$ for some integer $N$. In these cases, the matrix $\Psi'$ should be defined over the respective ring, and the isomorphy type of the respective moduli space $\mathsf{T}$ will depend only on the congruency class of the matrix $\Psi'$ over the ring under consideration.
\end{rem}

\begin{teo}\label{quasiaffine}
a patch of $\mathsf{T}$ is a quasi-affine complex variety.
\end{teo}
\begin{proof}

By Theorem \ref{compatible} or its preceeding remarks there is a non-empty Zariski subset $V$ of $\mathsf{S}$ there are algebraic sections $\omega=(\omega_1,\ldots,\omega_5)$ on $V$, such that they form a frame for the cohomology bundle, and are compatible with the Hodge filtration. For $i,j=1,\ldots,5,$ let us define regular functions $b_{ij}\in\mathcal{O}_U(V)$ by means of the intersection pairing:

\begin{equation*} b_{ij}(s):=\langle \omega_i(s),\omega_j(s)\rangle=\frac{1}{(2\pi i)^2}\int_{X(s)} \omega_i(s)\cup\omega_j(s)\in \mathcal{O}_U(V),\,\,  s\in U.
\end{equation*}

Any other tuple of sections $\alpha=(\alpha_1,\ldots,\alpha_5)$ on $U$ forming a frame for the cohomology bundle is obtained from $\omega$ in the form $\alpha=S\omega,$ where $S$ is the change of basis matrix. For $\alpha$ to satisfy the conditions in Definition \ref{tmoduli1}, $S$ must be of the form
\begin{equation*}
S =
\begin{bmatrix}
1 & 0 & 0 & 0 & 0\\
s_{21} & s_{22} & s_{23} & s_{24} & 0\\
s_{31} & s_{32} & s_{33} & s_{34} & 0\\
s_{41} & s_{42} & s_{43} & s_{44} & 0\\
s_{51} & s_{52} & s_{53} & s_{54} & s_{55}

\end{bmatrix},
\end{equation*}

and satisfy $S\Omega S^T=\Psi$, where $\Omega=[b_{ij}]\in \textsf{Mat}_{5}(\mathcal{O}_U(V)).$ Therefore, $\mathsf{T}$ is isomorphic to the spectrum of the $\mathbb{C}$-algebra
\begin{equation*}
\frac{\mathcal{O}_U(V)[s_{ij}]}{\mathcal{I}},\end{equation*}
where $\mathcal{I}$ is the ideal generated by the relations $S\Omega S^T=\Psi.$ So, we conclude that $\mathsf{T}$ is a quasi-affine complex variety, and $\Gamma(\mathcal{O}_{\mathsf{T}})=\mathbb{C}[a,b,c,d]\otimes_{\mathbb{C}}\frac{\mathbb{C}[s_{ij}]}{\mathcal{I}}.$
\end{proof}

\begin{rem}
Observe that by the previous proof, Proposition \ref{invariance} extends in a local chart of the moduli spaces $\mathsf{T}_{\Psi'}$ and $\mathsf{T}_{\tilde{\Psi}'}$ to a in isomorphism of varieties over $\mathbb{C}$. But, as mentioned in Remark \ref{notinvariant}, the moduli spaces may not be isomorphic over smaller rings.  
\end{rem}

Working out, we find that to find local coordinates for $\mathsf{T}$ it is enough to find three independent parameters for $S^{1,1}$ in $S^{1,1}\Omega^{1,1}(S^{1,1})^T=\Psi'$, and choose also $S^{1,0}$ as independent parameters since we have $S^{2,2}\Omega^{2,0}=1$ and

\begin{equation}
(S^{2,1})^T=-(S^{1,1}\Omega^{1,1})^{-1}(S^{1,0}\Omega^{0,2}+S^{1,1}\Omega^{1,2})S^{2,2}
\end{equation}
\begin{equation}(S^{2,0})^T=\frac{-1}{\Omega^{2,0}}(-(S^{2,1}\Omega^{1,1}+S^{2,2}\Omega^{2,1})(S^{1,1}\Omega^{1,1})^{-1}(S^{1,0}\Omega^{0,2}+S^{1,1}\Omega^{1,2})+S^{2,0}\Omega^{0,2}+S^{2,1}\Omega^{1,2}+S^{2,2}\Omega^{2,2})
\end{equation}
\subsection{The algebraic group $\mathsf{G}$}
In this section we compute the dimension of the moduli space $\mathsf{T}$ by introducing an algebraic group $\mathsf{G}$ acting on it in such a way that two enhanced K3 surfaces belong to the same orbit if, and only if, the underlying K3 surfaces have the same complex structure. 

For the following definition, recall the definition of $\Psi$ in Equation \ref{psi}.

\begin{df}\label{alggroup1} We define the complex algebraic group
\begin{equation*}
\mathsf{G}_{\Psi}=\{g\in \mathsf{Mat}_{5}(\mathbb{C})\, | \, g^T\Psi g=\Psi\text{ and } g^{T} \text{respects Hodge filtration} \}.
\end{equation*}
\end{df}
In the previous definition, $g^T$ respects Hodge filtration if and only if $g^T$ is of the form
\begin{equation*}
\begin{bmatrix}
*_{1\times 1} & 0 & 0 \\
* & *_{3\times 3} & 0\\
* & * & *_{1\times 1}\\
\end{bmatrix}\in \mathsf{Mat}_5(\mathbb{C}).
\end{equation*}

We have the following analog of Proposition \ref{invariance}. We keep the hypothesis of that proposition.

\begin{prop}\label{groupinvariance}
$\mathsf{G}_{\Psi}$ is isomorphic to $\mathsf{G}_{\tilde{\Psi}}$ as algebraic groups over $\mathbb{C}$.
\end{prop}
\begin{proof}
Let $P$ be as in the proof of Proposition \ref{invariance}. Let us define \begin{equation*}
Q:= \begin{bmatrix}
1 & 0 & 0 \\
0 & P & 0\\
0 & 0 & 1 
\end{bmatrix}.
\end{equation*}
$Q$ is non-singular complex matrix, and $Q^T\Psi Q=\tilde{\Psi}$. Then, $g\mapsto Q^TgQ$ is the desired isomorphism. 
\end{proof}
\begin{rem}
The matrix $\Psi'$ in the definition of $\Psi$ (see Equation \ref{psi}) might be defined over over an smaller ring than $\mathbb{C}$. In this case, the algebraic group $\mathsf{G}_{\Psi}$ would also be defined over the smaller ring. On the other hand, the isomorphy type of $\mathsf{G}_{\Psi}$ would depend on the congruence class of the matrix $\Psi'$ over the ring under consideration.
\end{rem}

One of the good reasons that leave us to define this group is the following proposition

\begin{prop} $\mathsf{G}$ acts algebraically on the right on the moduli space $\mathsf{T}$ by means of 
\begin{equation*}
(X,\alpha)\cdot g=(X,g^{T}\alpha).
\end{equation*}
Furthermore, $\mathsf{T}/ \mathsf{G}$ is isomorphic to the moduli space $\mathcal{M}$.
\end{prop}
\begin{proof}
This action is algebraic by using Theorem \ref{quasiaffine}. The remaining assertions are immediate.
\end{proof}

\subsection{Explicit computations for the lattice $N$}

Motivated by the computations performed in the next section, our convenient choice of intersection matrix in the case of the lattice $N$ will be 
\begin{equation}\label{choice}
\Psi:=
\begin{bmatrix}
0 & 0 & 0 & 0 & 1\\
0 & 0 & 0 & 1 & 0 \\
0 & 0 & -\frac{1}{2} & 0 & 0 \\
0 & 1 & 0 & 0 & 0 \\
1 & 0 & 0 & 0 & 0 \\
\end{bmatrix}. 
\end{equation}

\begin{prop}\label{lie}
$\mathfrak{g}:=\text{Lie}(\mathsf{G})$ is isomorphic to the Lie subalgebra of $\mathfrak{o}_{5,\Psi}(\mathbb{C})$ generated freely by 
\begin{equation*}
\mathfrak{g}_1:=
\begin{bmatrix}
0 & 1 & 0 & 0 & 0\\
0 & 0 & 0 & 0 &  0\\
0 & 0 & 0 & 0 & 0 \\
0 & 0 & 0 & 0 & -1 \\
0 & 0 & 0 & 0 & 0 \\
\end{bmatrix}, 
\mathfrak{g}_2:= \begin{bmatrix}
0 & 0 & 1 & 0 & 0\\
0 & 0 & 0 & 0 & 0 \\
0 & 0 & 0 & 0 & 2  \\
0 & 0 & 0 & 0 & 0 \\
0 & 0 & 0 & 0 & 0 \\
\end{bmatrix}, 
\mathfrak{g}_3:=\begin{bmatrix}
0 & 0 & 0 & 1 & 0\\
0 & 0 & 0 & 0 & -1 \\
0 & 0 & 0 & 0 & 0 \\
0 & 0 & 0 & 0 & 0 \\
0 & 0 & 0 & 0 & 0 \\
\end{bmatrix},
\end{equation*}
\begin{equation*}
\mathfrak{g}_4:=
\begin{bmatrix}
0 & 0 & 0 & 0 & 0\\
0 & 1 & 0 & 0 & 0 \\
0 & 0 & 0 & 0 & 0 \\
0 & 0 & 0 & -1 & 0 \\
0 & 0 & 0 & 0 & 0 \\
\end{bmatrix}, 
\mathfrak{g}_5:= \begin{bmatrix}
0 & 0 & 0 & 0 & 0\\
0 & 0 & 1 & 0 & 0 \\
0 & 0 & 0 & 2 & 0\\
0 & 0 & 0 & 0 & 0\\
0 & 0 & 0 & 0 & 0 \\
\end{bmatrix}, 
\mathfrak{g}_6:=\begin{bmatrix}
0 & 0 & 0 & 0 & 0\\
0 & 0 & 0 & 0 & 0 \\
0 & 1 & 0 & 0 & 0 \\
0 & 0 & \frac{1}{2} & 0 & 0 \\
0 & 0 & 0 & 0 & 0 \\
\end{bmatrix},
\end{equation*}
\begin{equation*}
\mathfrak{g}_0:=
\begin{bmatrix}
1 & 0 & 0 & 0 & 0\\
0 & 0 & 0 & 0 & 0 \\
0 & 0 & 0 & 0 & 0 \\
0 & 0 & 0 & 0 & 0 \\
0 & 0 & 0 & 0 & -1 
\end{bmatrix}.
\end{equation*} 
\end{prop}

For future use, we compute its Lie brackets

\begin{center}
\begin{tabular}{|c||c|c|c|c|c|c|c|}\hline
 & $\mathfrak{g}_1$ & $\mathfrak{g}_2$ & $\mathfrak{g}_3$ & $\mathfrak{g}_4$ & $\mathfrak{g}_5$ & $\mathfrak{g}_6$ & $\mathfrak{g}_0$ \\ \hline\hline
 
$\mathfrak{g}_1$ & 0 & 0 & 0 & $
\mathfrak{g}_2$ & $\mathfrak{g}_3$ & 0 & $-\mathfrak{g}_1$\\ \hline
 
$\mathfrak{g}_2$ & 0 & 0 & 0 & $-\mathfrak{g}_1$ & $0$ & $\mathfrak{g}_3$ & $-\mathfrak{g}_2$\\ \hline
 
$\mathfrak{g}_3$ & 0 & 0 & 0 & $0$ & $-\mathfrak{g}_1$ & $-\mathfrak{g}_2$ & $-\mathfrak{g}_3$ \\ \hline
 
$\mathfrak{g}_4$ & $-\mathfrak{g}_2$ & $\mathfrak{g}_1$ & $0$ & 0 & $\mathfrak{g}_5$ & $-\mathfrak{g}_6$ & 0 \\ \hline
 
$\mathfrak{g}_5$ & $-\mathfrak{g}_3$ & $0$ & $\mathfrak{g}_1$ & $-\mathfrak{g}_5$ & 0 & $\mathfrak{g}_4$ & 0 \\ \hline
 
$\mathfrak{g}_6$ & $0$ & $-\mathfrak{g}_3$ & $\mathfrak{g}_2$ & $\mathfrak{g}_6$ & $-\mathfrak{g}_4$ & 0 & 0 \\ \hline
 
$\mathfrak{g}_0$ & $\mathfrak{g}_1$ & $\mathfrak{g}_2$ & $\mathfrak{g}_3$ & 0 & 0 & 0 & 0 \\ \hline
  
\end{tabular}
\end{center}

\begin{proof}
By standard Lie group arguments, we have that
\begin{equation*}
\text{Lie}(\mathsf{G})=\{x\in \mathsf{Mat}_{5}(\mathbb{C})|\, x^{T} \text{ respects Hodge filtration and } x^T\Psi+\Psi x=0\}.\end{equation*} Let $x\in \text{Lie}(\mathsf{G})$. Let us write
\begin{equation*}
x=\begin{bmatrix}
a & b & c \\
0 & d & e\\
0 & 0 & f\\
\end{bmatrix}.
\end{equation*}
Then, condition $x^T\Psi+\Psi x=0$ is equivalent to conditions $f=-a, d=-d^T,e=-b^T, c=-c$. Therefore, $\text{Lie}(\mathsf{G})$ is isomorphic to the Lie subalgebra of $\mathfrak{gl}_5(\mathbb{C})$ consisting of elements of the form

\begin{equation*}
x=\begin{bmatrix}
a & b & 0 \\
0 & c & -b^T\\
0 & 0 & -a\\
\end{bmatrix},
\end{equation*}
such that $c$ is antisymmetric.
\end{proof}

\begin{prop}
The Lie subalgebra of $\text{Lie}(\mathsf{G})$ generated by $W_1,W_2,W_3$ is isomorphic to $\mathfrak{sl}_2(\mathbb{C})$. 
\end{prop}

\begin{proof}
Let $L=\langle W_1,W_2,W_3\rangle$. Then, $L$ is a complex three-dimensional Lie algebra and, by looking at the previous table, it satisfies $L=L'$. The classification of three dimensional Lie algebras implies that $L\cong \mathfrak{sl}_2(\mathbb{C})$. 
\end{proof}

\begin{cor}
The moduli space $\mathsf{T}$ has dimension $10$.
\end{cor}
\begin{proof}
By the previous proposition, $dim(\mathsf{G})=7$. Therefore, $dim(\mathsf{T})=dim(\mathcal{M})+dim(\mathsf{G})=3+7=10.$
\end{proof}

\begin{df}
The AMSY-Lie algebra $\mathfrak{G}$ associated to the Clingher-Doran family of $N$-polarized K3 surfaces is the Lie subalgebra of $\mathfrak{gl}_5(\mathbb{C})$ generated by $\text{Lie}(\mathsf{G})$ and $V^T_1,V^T_2,V^T_3$.
\end{df}

\begin{teo} $\mathfrak{G}$ is isomorphic to $\mathfrak{sp}_4(\mathbb{C})$.
\end{teo}

\begin{proof}
By the proof of Proposition \ref{lie}, it follows that $\langle \text{Lie}(\mathsf{G}), V^T_1,V^T_2,V^T_3\rangle$ is  equal to $\mathfrak{o}_{5,\Psi}(\mathbb{C}).$ Since $\Psi$ is symmetric and nondegenerate, we have that $\mathfrak{o}_{5,\Psi}(\mathbb{C})\cong \mathfrak{so}_{5}(\mathbb{C}).$ We conclude by using the classical fact $\mathfrak{so}_{5}(\mathbb{C})\cong \mathfrak{sp}_{4}(\mathbb{C})$.
\end{proof}

\subsection{Computations for arbitrary lattice polarization}

First, let us note that, by Proposition \ref{groupinvariance}, the algebraic group $\mathsf{G},$ and the $\mathsf{AMSY}$-Lie algebra $\mathfrak{G}$ depend only (over $\mathbb{C}$) on the rank of the polarization under consideration. So, let us denote by $\mathsf{G}_k$ and $\mathfrak{G}_k$ the corresponding objects when a polarization of rank $1\leq k\leq 20$ is considered.

In \cite{alimvogrin}, it was asked for a complete classification of these objects. With respect to this question we have the following theorem. 
\begin{teo} 
\begin{enumerate} 
\item The radical of $\mathsf{Lie(G})_k$ has dimension $21-k$;

\item (Computation of the Levi factor) $\mathfrak{Lie}(G_k)/\mathfrak{rad(Lie}(G_k))\cong \mathfrak{so}_{20-k}(\mathbb{C})$;
\item $\mathsf{AMSY}_k$ is the lie subalgebra of $\mathfrak{gl}_{22-k}(\mathbb{C})$ generated by $\mathfrak{Lie}(G_k)$ and $\{g^T|\, g\in\mathfrak{nilrad(Lie}(G_k))\}$ (the transpose of the nilradical).

\item $\mathsf{AMSY_k}\cong \mathfrak{so}_{22-k}(\mathbb{C})$.
\end{enumerate}
\end{teo}
\begin{proof}
Since we are working over $\mathbb{C}$, Proposition \ref{groupinvariance} implies that we can suppose that the intersection matrix is of the form

\begin{equation*}\label{psi2}
\Psi:= \begin{bmatrix}
0 & 0 & 1 \\
0 & I_{20-k} & 0\\
1 & 0 & 0 
\end{bmatrix}.
\end{equation*}.
By standard Lie group arguments, we have that
\begin{equation*}
\mathfrak{Lie}(G_k)=\{x\in \mathsf{Mat}_{22-k}(\mathbb{C})|\, x^{T} \text{ respects Hodge filtration and } x^T\Psi+\Psi x=0\}.\end{equation*} Let $x\in \mathfrak{Lie}(G_k)$. Let us write\
\begin{equation*}
x=\begin{bmatrix}
a & b & c \\
0 & d & e\\
0 & 0 & f\\
\end{bmatrix}.
\end{equation*}
Then, condition $x^T\Psi+\Psi x=0$ is equivalent to conditions $f=-a, d=-d^T,e=-b^T, c=-c$. Therefore, $\mathfrak{Lie}(G_k)$ is isomorphic to the Lie subalgebra of $\mathfrak{gl}_{22-k}(\mathbb{C})$ consisting of elements of the form

\begin{equation}\label{description}
x=\begin{bmatrix}
a & b & 0 \\
0 & c & -b^T\\
0 & 0 & -a\\
\end{bmatrix},
\end{equation}
such that $c$ is antisymmetric. Using this description of $\mathfrak{Lie}(G_k)$. we can easily conclude our theorem.
Let us observe that the subalgebra of $\mathfrak{Lie}(G_k)$ generated by the independent entries given by $a$ and $b$ is an ideal, which we denote by $I$. Furthermore, it is solvable since it is a Lie subalgebra of $\mathfrak{gl}_{22-k}(\mathbb{C})$ consisting of upper-triangular matrices. To conclude that it is the radical of $\mathfrak{Lie}(G_k)$, we must show that it is maximal. Let $J$ be a solvable ideal of $\mathfrak{Lie}(G_k)$ containing  $I$. Then $J/I$ is a solvable Lie algebra which can be identified with an ideal of the subalgebra of $\mathfrak{Lie}(G_k)$ generated by the independent entries given by $c$, where $c$ is antysimmetric. On the other hand, this algebra is isomorphic to $\mathfrak{so}_{20-k}(\mathbb{C})$, which is semisimple. Since the only solvable ideal of a semisimple Lie algebra is $\langle 0 \rangle$, we conclude that $I=J$. This completes the computation of $\mathfrak{rad(Lie}(G_k))$. As a byproduct, we also conclude $2.$
To prove $3.$ and $4.$, let us observe that by description \ref{description} of the elements of $\mathfrak{Lie}(G_k)$, we can conclude that $\mathfrak{nilrad(Lie}(G_k))$ consists of the independent entries of $b$, which implies $3.$ Using this description, we have that the elements of $\{g^T|\, g\in\mathfrak{nilrad(Lie}(G_k))\}$ can be written in the form 

\begin{equation*}
x=\begin{bmatrix}
0 & 0 & 0 \\
b^T & 0 & 0\\
0 & -b & 0\\
\end{bmatrix},
\end{equation*}
Such an $x$ satisfies  $x^T\Psi+\Psi x=0.$ Therefore, the Lie algebra $\mathsf{AMSY}_k$ is a Lie subalgebra of $\mathfrak{o}_{22-k,\Psi}(\mathbb{C})$ of dimension $1+\frac{(20-k)(20-k-1)}{2}+2(20-k)=\frac{(22-k)(22-k-1)}{2}=dim(\mathfrak{o}_{22-k,\Psi}(\mathbb{C}))$. Therefore, $\mathsf{AMSY}_k=\mathfrak{o}_{22-k,\Psi}(\mathbb{C})\cong \mathfrak{so}_{22-k}(\mathbb{C})$.
\end{proof}

As a a corollary to the previous theorem, we obtain Theorem 4.2 in \cite{alimvogrin}:
\begin{cor}
$\mathfrak{G}_{18}\cong \mathfrak{sl}_2(\mathbb{C})\oplus \mathfrak{sl}_2(\mathbb{C}).$
\end{cor}

%% file: s5.tex
In \cite{g1}, for a given compact Kähler manifold, Griffiths defined its classical period domain $\mathsf{D}$: the space in which periods of forms belonging to the first non-trivial piece of the Hodge filtration lived; for the case of smooth complex projective varieties, these periods corresponded to periods of forms of the first kind. The objective of this construction was to study the variation of the period matrix in a family, via the classical period map, which should give information about the variation of complex structure on the fibers. To be well-defined, this map needed to take into account the action of the monodromy $\Gamma$ on the period domain, leading to what Griffiths called the \textit{modular variety} $\mathsf{M}=\mathsf{D}/\Gamma$. Name inspired from the construction of automorphic forms by means of the Baily-Borel compactification of $\mathsf{M},$ whenever $\mathsf{D}$ was a hermitian symmetric domain \cite{baily}.  On the other hand, the last condition was not satisfied by the very interesting case of mirror quintics, from which many new $q$-expansions where appearing in the context of physics, which opened the question of an automorphic form theory for this case.

In relation to the previous problem, and a geometric interpretation of the Ramanujan equations between modular forms, Movasati in \cite{pascal,mov17,mov20} systematically used the periods of the whole primitive middle cohomology, instead of only the ones coming from the first piece of the Hodge filtration, to define a generalized period map on moduli spaces of enhanced varieties of a fixed type. This generalized period map is in general locally injective. In this section, we prove that the generalized period map for the Clingher-Doran family is a biholomorphism, which is essentially a consequence of the Torelli theorem for K3 surfaces. 
In \cite{mov13}, Movasati asked for an algebraization of the generalized period domain for the case of principally polarized abelian surfaces. The previous result can be considered as answer to that demand. It is good to observe that T. Fonseca solved this algebraization problem for principally polarized varieties of arbitrary dimension \cite{fon}.

Finally, to obtain meromorphic quasimodular forms on the Siegel half space of genus two, we define the $\mathbf{T}$-map for $N$-polarized K3 surfaces, which basically amounts to constructing, in a holomorphic way,  a polarized Hodge structure of type $(1,3,1)$ out of a given complex structure, which is compatible with the original complex structure.

\subsection{Intersection product in homology}

Let us fix a base point $0\in\mathsf{T}$. Poincaré duality gives us an isometry 
\begin{equation}\label{universal}
H^2(X_0,\mathbb{Z})\cong H_2(X_0,\mathbb{Z}).
\end{equation}

Recall that $X_0$ has an $N$-polarization $\iota:N\rightarrow H^2(X_0,\mathbb{Z}).$ Let us denote by $\iota(N)^{Pd}$ the image of $\iota(N)$ under the previous isomorphism. Let $H_2(X_0,\mathbb{Z})_{\iota}:=(\iota(N)^{Pd})^{\perp}$ be the orthogonal complement taken with respect to the intersection product. Observe that this is possible since the embedding $i$ is primitive and $H_2(X_0,\mathbb{Z})$ is torsion-free. Furthermore, we have the isometries
\begin{equation}
H_2(X_0,\mathbb{Z})_{\iota}\cong H_2(X_0,\mathbb{Z})/\iota(N)^{Pd}\cong H\oplus H\oplus\langle -2\rangle=N^{\perp}.
\end{equation}

Let $\delta_0=(\delta_1,\ldots,\delta_5)^T$ be a basis of $H_2(X_0,\mathbb{Z})_{\iota}$ with fixed intersection matrix 

\begin{equation*}
[\langle \delta_i,\delta_j\rangle]=\Psi^{-1}=
\begin{bmatrix}
0 & 0 & 0 & 0 & 1\\
0 & 0 & 0 & 1 & 0 \\
0 & 0 & -2 & 0 & 0 \\
0 & 1 & 0 & 0 & 0 \\
1 & 0 & 0 & 0 & 0 \\
\end{bmatrix}. 
\end{equation*}

We will often also use the Poincaré dual basis $\delta^{Pd}=(\delta_1^{Pd},\ldots,\delta_5^{Pd})^T$ of $H^2(X,\mathbb{Z})_{\iota}$.

\subsection{Generalized period domain and period map}

In this section we define the generalized period domain, i.e., the target space for the generalized period map defined on our moduli space $\mathsf{T}$. This period domain contains more information than the classical period domain since, besides the periods of holomorphic differentials, it contains the periods corresponding to the entire trascendental cohomology. To motivate its definition, we include the following proposition.

\begin{prop}\label{pmatrix} Let $(X,\iota,\alpha_1,\ldots,\alpha_5)$ be an $N$-polarized K3 surface, and $\delta=(\delta_1,\ldots,\delta_5)^T$ a basis of $H_2(X,\mathbb{Z})_{\iota}$ with intersection matrix equal to $\Psi^{-1}$.  Let $P:=[\int_{\delta_i}\alpha_j]$. Then, $P\in\mathsf{GL}_5(\mathbb{C})$, $\Psi=P^{T}\Psi P$, and $(P^1)^T\Psi\overline{P^1}>0$.
\end{prop}

\begin{proof}
First, we proof that $\Psi=P^{T}\Psi P$. Since $\alpha=P^T\Psi\delta^{pd},$ then $\Psi=\alpha\alpha^T=P^T\Psi\delta^{pd}(\delta^{pd})^T\Psi^TP=P^T\Psi\Psi^{-1}\Psi^{T}P=P^T\Psi P.$ In the last equality we used that $\Psi$ is symmetric. Observe that, since $det(\Psi)\neq 0$, this implies that $det(P)\neq 0$. Finally, the last assertion in this proposition follows from the Riemann bilinear relations.
\end{proof}

The previous proposition suggests the following definition, which can be found in its widest generality at \cite[Chapter 8]{mov20}.
\begin{df} The \textbf{manifold of period matrices} is 
\begin{equation*}\mathsf{\Pi}=\{P\in \mathsf{GL}_5(\mathbb{C})|\, P^{T}\Psi P=\Psi \text{ and } (P^1)^T\Psi\overline{P^1}>0\}.
\end{equation*}
\end{df}

\begin{prop}\label{tangent}
$\mathsf{\Pi}$ is a smooth complex manifold of dimension $10$. Furthermore, for any $P\in \mathsf{\Pi},$ $T_P\mathsf{\Pi}=\{X\in\mathsf{Mat}_5(\mathbb{C})|\, P^T\Phi^{-1} X+X^T\Phi^{-1} P=0\}$.
\end{prop}

\begin{proof}
Let us define the holomorphic map $F:\mathsf{GL}_5(\mathbb{C})\rightarrow \mathsf{Mat}_5(\mathbb{C}),\, F(P)=P^{T}\Phi^{-1} P$. This map is $\mathsf{GL}_5(\mathbb{C})$-equivariant with respect to the right actions given by multiplication on the right on $\mathsf{GL}_5(\mathbb{C})$, and $X\cdot A:=A^TX A$ on $\mathsf{Mat}_5(\mathbb{C})$. Furthermore, since the first action is transitive, we conclude that $F$ has constant rank. Therefore, $F^{-1}(\Psi)$ is a properly embedded complex submanifold of $\mathsf{GL}_5(\mathbb{C})$. Since $U:=\{P\in\mathsf{GL}_5(\mathbb{C})|(P^1)^T\Phi^{-1}\overline{P^1}>0\}$ is open in $\mathsf{GL}_5(\mathbb{C})$, then $\mathsf{\Pi}=F^{-1}(\Psi)\cap U$ is a open submanifold of $F^{-1}(\Psi)$. Therefore, for any given $P\in \mathsf{\Pi},$ we have $T_P \mathsf{\Pi}=T_P  F^{-1}(\Psi)=ker(F_{*,P})=\{X\in\mathsf{Mat}(5,\mathbb{C})|\, P^T\Phi^{-1} X+X^T\Phi^{-1} P=0\}$. Finally, since $F$ has constant rank, to compute the dimension of $\mathsf{\Pi}$ it suffices to compute $dim(T_{I_5}F^{-1}(\Phi^{-1}))$. Since $X\mapsto \Phi^{-1} X$ defines an isomorphism $T_{I_5}F^{-1}(\Phi^{-1})\rightarrow \mathsf{Skew}_5(\mathbb{C})$ of complex vector spaces, we have $dim(T_{I_5}F^{-1}(\Phi^{-1}))=dim(\mathsf{Skew}_5(\mathbb{C}))=10$.
\end{proof}

\begin{df}\label{local}
Let $\mathcal{H}$ be the local system of $\mathbb{Z}$-modules on $\mathsf{T}$ formed by the set of trascendental homology groups $H_2(X_t,\mathbb{Z})_{\iota}$ for $t\in \mathsf{T}$.
\end{df}

Using proposition \ref{pmatrix}, we would like to define the generalized period map $\mathsf{T}\rightarrow \mathsf{\Pi}$. To do this, we would need to find a continuous global section $\delta$ of $\mathcal{H}$. Whether or not we can do this, it depends on the monodromy of the family under consideration. In the case of the Clingher-Doran family has non-trivial monodromy, we cannot choose such a global section. 

To overcome this, let $\pi:\tilde{\mathsf{T}}\rightarrow \mathsf{T}$ be the universal cover of the analytic space $\mathsf{T}$, and let $\tilde{\mathcal{H}}:=\pi^*\mathcal{H}$ be the pullback to $\tilde{\mathsf{T}}$ of $\mathcal{H}$. Since $\tilde{\mathsf{T}}$ is simply connected, $\tilde{\mathcal{H}}$ is trivial. Therefore, we have an isomorphism $m:\tilde{\mathcal{H}}\rightarrow N^{\perp}\times \tilde{\mathsf{T}}$. This isomorphism is an isometry fiberwise. Let us denote by $m(\tilde{t}):H_2(X_{\pi(\tilde{t})},\mathbb{Z})_{\iota}\rightarrow N^{\perp}$ the isometry induced by $m$ at $\tilde{t}\in \tilde{\mathsf{T}}$.

\begin{df}Let $\rho:\pi(\mathsf{T};0)\rightarrow Aut(H_2(X,\mathbb{Z})_{\iota},\langle \cdot,\cdot \rangle)$ be the monodromy representation of the Clignher-Doran family. 
\end{df}
Recall that $\pi(\mathsf{T};0)$ acts on $\tilde{\mathsf{T}}$ by deck transformations. 
\begin{prop}\label{monodromy} For every $\gamma\in \pi(\mathsf{T};0)$ and $\tilde{t}\in\tilde{\mathsf{T}}$, we have $m(\gamma\cdot \tilde{t})=\rho(\gamma)\circ m(\tilde{t}).$
\end{prop}
\begin{proof}
This follows from the fact that $\tilde{\mathcal{H}}$ is a local system. 
\end{proof}

Since the global trivialization of $\tilde{\mathcal{H}}$ preserves the intersection product in each fiber, we can find a continuous global section $\delta$ of $\tilde{\mathcal{H}}$ such that, for each $\tilde{t}\in \mathsf{T}$, $\delta(\tilde{t})=(\delta_1(\tilde{t}),\ldots,\delta_5(\tilde{t}))^T$ is a basis of $H_2(X_{\pi(\tilde{t})},\mathbb{Z})_{\iota}$ with intersection matrix equal to $\Phi$. Let us fix such a $\delta$.
 
\begin{df}. The \textbf{generalized period map} $\tilde{\mathcal{P}}_{\delta}:\tilde{\mathsf{T}}\rightarrow \mathsf{\Pi}$ associated to $\delta$ is defined by 
\begin{equation*}
\tilde{\mathcal{P}}_{\delta}(\tilde{t})=[\int_{\delta_i(\tilde{t})}\alpha_j(\pi(\tilde{t}))].
\end{equation*} 
\end{df}

By means of $\delta$, we can identify the group $\Gamma:=\{A\in \mathsf{GL}_5(\mathbb{Z})|\, A^T\Phi A=\Phi  \}$ with $\mathsf{O}(H_2(X,\mathbb{Z})_{\iota})$, via the isomorphism $A\mapsto (f:\delta\mapsto A^T\delta).$

\begin{df} $\Gamma$ acts on the right of $\mathsf{\Pi}$ by means of $A\cdot P=A^{-T}P$ for every $P\in\mathsf{\Pi}$ and $A\in \Gamma$.
\end{df}
\begin{prop}
$\tilde{\mathcal{P}}_{A\cdot \delta}=A^{-1}\cdot\tilde{\mathcal{P}}_{\delta}$.
\end{prop}
\begin{proof}
Let $A=[a_{ij}]$. Then $\tilde{\mathcal{P}}_{A\cdot \delta}(\tilde{t})=[\int_{\sum_k a_{ki}\delta_k(\tilde{t})}\alpha_j(\pi(\tilde{t}))]=[\sum_k (\int_{\delta_k(\tilde{t})}\alpha_j(\pi(\tilde{t})))a_{ki}]= A^T\tilde{\mathcal{P}}_{\delta}(\tilde{t})=A^{-1}\cdot \tilde{\mathcal{P}}_{\delta}.$
\end{proof}

The previous proposition allows us to give the following definition, which corresponds to \cite[Definition 8.5]{mov20}.

\begin{df}
The \textbf{generalized period map} $\mathcal{P}:\mathsf{T}\rightarrow \Gamma\backslash \mathsf{\Pi}$ is defined to be the quotient of $\tilde{\mathcal{P}}_{\delta}$ by the action of $\pi(\mathsf{T};0)$.
\end{df}

Observe that the previous map is independent of the sections $\delta$ since any two such sections are obtained by the action of an element of $\Gamma$. Now, we aim to prove that the target space of the generalized period domain has the structure of a complex analytic space.

\begin{df}
$\mathsf{G}$ acts on the right of $\mathsf{\Pi}$ by matrix multiplication.
\end{df}

\begin{prop}\label{comp}
$\Gamma$ and $\mathsf{G}$ act properly and discontinuously on $\mathsf{\Pi}$. Both actions are compatible in the following sense: $A\cdot(P\cdot g)=(A\cdot P)\cdot g$. 
\end{prop}

\begin{proof}
Since $\Gamma$ acts continuously on $\mathsf{\Pi}$, it suffices to prove that if $(P_i)$ and $(A_i)$ are sequences in $\mathsf{\Pi}$ and $\Gamma$, respectively,  such that both $(P_i)$ and $(A_i\cdot P_i)$ converge, then $(A_i)$ has a convergent subsequence. Since $\mathsf{\Pi}\subset \mathsf{GL}_5(\mathbb{C})$ and taking the inverse of a matrix is a continuous operation, the assertion for $\Gamma$ follows by observing that $\Gamma$ is a closed and discrete subgroup of $\mathsf{GL}_5(\mathbb{C})$. The proof for $\mathsf{G}$ is analogous. Compatibility is immediate since matrix multiplication is associative.
\end{proof}
\begin{teo}Let $X$ be an analytic space, $G$ a group acting properly and discontinuously on $X$ by biholorphisms, and $\rho:X\rightarrow X/G$ the quotient map. The sheaf $\mathcal{O}_{X/G}$ defined on $X/G$ by 

\begin{equation*}
\mathcal{O}_{X/G}(U)=\{f:U\rightarrow \mathbb{C}\, |\, f\circ \rho\in \mathcal{O}_X(\rho^{-1}(U))\}
\end{equation*}
defines a structure of analytic space on $X/G$.
\end{teo}
\begin{proof}
See \cite{car}
\end{proof}

\begin{cor} The topological spaces $\Gamma\backslash \mathsf{\Pi}$ and $\mathsf{\Pi}/G$ have the structure of analytic spaces.
\end{cor}
\begin{cor}
$(\Gamma\backslash \mathsf{\Pi})/G$ and $\Gamma\backslash (\mathsf{\Pi}/G)$ are biholomorphic.
\end{cor}
\begin{df} The \textbf{generalized period domain} is the space $$\mathsf{U}=\Gamma\backslash \mathsf{\Pi}.$$
The \textbf{classical period domain} is the space
$$\mathsf{D}:=\mathsf{\Pi}/G.$$
\end{df}

\begin{prop}
$\mathcal{P}$ is holomorphic.
\end{prop}

\begin{proof}
Since this question is local, we can assume that we are working on a simply-connected open subset $\mathcal{O}$ of $\mathsf{T}$, and, therefore, we can omit the monodromy $\Gamma$. The proof of the holomorphicity of $\mathcal{P}:\mathcal{O}\rightarrow \mathsf{\Pi}$ is exactly as in \cite{griffiths} since the same proof applies for integration of non-holomorphic forms. 
\end{proof}

\begin{prop}\label{eq} 
$\mathcal{P}$ is $G$-equivariant: for every $t\in \mathsf{T}$ and $g\in G,$ we have $P(t\cdot g)=P(t)\cdot g.$
\end{prop}
\begin{proof}
Observe that $[\int_{\delta_i}(g^T\alpha)_{j}]=[\int_{\delta_i}\sum_k(g^T)_{jk}\alpha_k]=[\sum_k(\int_{\delta_i}\alpha_k)g_{kj}]=[\int_{\delta_i}\alpha_j]\cdot g$. Then, the proposition follows from this, and compatibility in Proposition \ref{comp}.
\end{proof}

The following consequence of Torelli theorem for K3 surfaces will be used to proof that $\mathcal{P}$ is a biholomorphism.

\begin{teo}
For a complex K3 surface $X$, the map $f\mapsto f^*$ induces an isomorphism  betweem the group of automoprhism of $X$, and the group of isometries $g\in \mathsf{O}(H^2(X,\mathbb{Z}))$ such that $g$ is a Hodge isometry, and sends some Kähler class to a Kähler class.
\end{teo}\label{torelli}
\begin{proof}
See \cite[page 339, Corollary 2.3]{huy}.
\end{proof}
\begin{teo}\label{bihol}
$\mathcal{P}$ is a biholomorphism.
\end{teo}\label{iso}
\begin{proof}
By Proposition \ref{eq}, $\mathcal{P}$ is $G$-equivariant. Therefore, we have a well-defined map $\mathcal{P}/G:\mathsf{T}/G\rightarrow \mathsf{U}/G$. The Global Torelli Theorem for polarized K3 surfaces implies that $\mathcal{P}/G$ is a biholomorphism. This statement, together with Proposition \ref{eq}, implies that $\mathcal{P}$ is onto. Now, we deal with injectivity of $\mathcal{P}$. By injectivity of $\mathcal{P}/G$, we only need to prove this for the following case: let $t_1=[(X,\alpha)]$ and $t_2=[(X,\beta)]\in\mathsf{T}$ be such that $\mathcal{P}(t_1)=\mathcal{P}(t_2)$. Let $\delta$ be a continuous local section of $\mathcal{H}$ (recall Definition \ref{local}) which is defined on $t_1$ and $t_2$. Let us write, by an abuse of notation, $\delta(t_1)=\delta$. Then, $\mathcal{P}(t_1)=\mathcal{P}(t_2)$ means that, for some $A\in\Gamma$, $[\int_{\delta_i}\alpha_j]=A^{-T}[\int_{\delta_i}\beta_j]$. Since $\Gamma$ is a group, $A^{-1}$ also belongs to it. Using $A^{-1}$, we are going to define an isometry of $H^2(X,\mathbb{Z})_{\iota},$ which extends to a Hodge isometry of $H^2(X,\mathbb{Z})$. By the previous considerations we can define $f\in \mathsf{O}(H^2(X,\mathbb{Z})_{\iota}),$ by $f(\delta^{Pd})=A^{-T}\delta^{Pd}$. We can extend $f$ to an isometry $g\in \mathsf{O}(H^2(X,\mathbb{Z}))$ by defining it to be the identity on $\iota(N)$. Let us observe that $g$ is in fact a Hodge isometry: $g_{\mathbb{C}}(\beta)=A^{-T}[\int_{\delta_i}\beta_j]\Phi^{-1}\delta^{Pd}=[\int_{\delta_i}\alpha_j]\Phi^{-1}\delta^{Pd}=\alpha.$ Since $i(N)$ contains Kähler classes, then $g$ sends some Kähler class to a Kähler class. Then, Theorem \ref{torelli} implies that there is an automorphism $\phi:X\rightarrow X$ such that the induced map in the second integer cohomology satisfies $\phi_{*}=g.$ The construction of $g$ implies that $\phi$ preserves the lattice polarization, and $\phi^*(\beta)=\alpha$. Therefore, $\mathcal{P}$ is a bijective holomorphism. \cite[Theorem 8.4.4]{grauert} implies $\mathcal{P}$ is a biholomorphism
\end{proof}

We observe that everything in this section extends \textit{mutatis mutandi} to an arbitrary quasi-ample polarizarion.

\begin{df} The moduli of \textbf{framed N-polarized K3 surfaces} $\mathbb{H}$ is the moduli space of pairs $(X,\delta),$ where $X$ is an $N$-polarized K3 surface, and $\delta$ is basis of the abelian group $H_2(X,\mathbb{Z})_{i}$ with intersection product equal to $\Phi$.  
\end{df}
\begin{prop}
$\mathbb{H}$ is isomorphic to the Griffiths-Dolgachev domain $\mathsf{D}$. 
\end{prop}

\begin{proof}
This follows from the fact that the period map $\mathcal{P}/G$ is an isomorphism.
\end{proof}

\subsection{The $\mathbf{T}$ and $\mathsf{t}$-maps}

From now on in this section, we will work will work with the Clingher-Doran family of $N$-polarized K3 surfaces. 

For the Clingher-Doran family, $\mathsf{D}=IV_{3}\coprod \overline {IV}_{3}$, which is biholomorphic to $\mathbb{H}_2\coprod\overline{\mathbb{H}_2}$.

\begin{df}
The {$\mathbf{T}$-map} of the Doran and Clingher family is the holomorphic map 
\begin{gather}\mathbf{T}:\mathbb{H}_2\rightarrow \mathsf{\Pi}, \notag\\
\begin{bmatrix}
\tau_1 & \tau_2 \\
\tau_2 & \tau_3
\end{bmatrix}\mapsto \begin{bmatrix}
\tau_2^2-\tau_1\tau_3 & -\tau_3 & -2\tau_2 & -\tau_1 & 1\\
\tau_3 & 0 & 0 & 1 & 0\\
\tau_2 & 0 & -1 & 0 & 0\\
\tau_1 & 1 & 0 & 0 & 0\\
1 & 0 & 0 & 0 & 0
\end{bmatrix}.\notag
\end{gather}
This map extends naturally to a holomorphic map $\mathbf{T}:\mathbb{H}_2\coprod\overline{\mathbb{H}_2}\rightarrow \mathsf{\Pi}.$
\end{df}

\begin{prop} The $\mathbf{T}$-map satisfies that the composition $\mathsf{D}\xrightarrow {\mathbf{T}} \mathsf{\Pi}\rightarrow \mathsf{\Pi}/ \mathsf{G}$ is the identity.
\end{prop}

\begin{proof}
Immediate from the definition of $\mathbf{T}$.
\end{proof}
From the previous proposition, we get immediately that $\tau$ is an injective holomorphic immersion. Indeed, we have:

\begin{prop}
$\mathbf{T}$ is a holomorphic embedding.
\end{prop}
\begin{proof}
Since $\mathsf{D}$ is Hausdorff, and $\mathbf{T}$ has a left-continuous inverse, we conclude that $\mathbf{T}$ is proper. This implies that $\mathbf{T}$ is an embedding.
\end{proof}

\begin{prop}
For every $P\in \mathsf{\Pi},$ there is an unique $g\in \mathsf{G}$ such that $Pg=\mathbf{T}([P])$. 
\end{prop}

\begin{proof}
By the previous propositions, $[\mathbf{T}([P])]=[P]$. Therefore, there is some  $g\in\mathsf{G}$ such that $Pg=\mathbf{T}([P]).$ It is unique since $\mathsf{G}$ acts freely on $\mathsf{\Pi}$. This last assertion follows from the fact that $\mathsf{\Pi}\subset \mathsf{GL}_5(\mathbb{C})$.
\end{proof}

The following proposition, which explains how to construct the $\mathsf{t}$-map from the $\tau$-map, are taken from \cite[Chapter 8]{mov20}.

\begin{prop}
There is an unique holomorphic map 

$$\mathsf{t}:\mathbb{H}\rightarrow \mathsf{T}$$ 

such that the following diagram commutes:

\begin{center}
\begin{tikzcd}
\mathbb{H} \arrow{r}{\mathsf{t}} \arrow[swap]{d}{pm} & \mathsf{T} \arrow{d}{P} \\
\mathsf{D}\arrow{r}{\mathbf{T}} & \mathsf{U}
\end{tikzcd}
\end{center}

and the composition $\mathbb{H}\xrightarrow {\mathsf{t}} \mathsf{T} \rightarrow \mathsf{T}/\mathsf{G}=\mathcal{M}$ is the canonical map $(X,\delta)\mapsto X$.

\end{prop}

\begin{proof}
Let $(X,\delta)\in\mathbb{H}$ and chosse an arbitrary enhacement $\alpha$ of $X$. By the previous proposition there is a unique $g\in\mathsf{G}$ such that $P(\delta,\alpha)g=\mathbf{T}([P(\delta,\alpha)])$. Define $\mathsf{t}(X,\delta)=(X,\alpha)\cdot g.$ This $\mathsf{t}$ satisfies the properties stated above. Indeed, $\mathsf{t}=P^{-1}\circ \mathbf{T}\circ pm$.
\end{proof}

\subsection{Construction of the $\mathbf{T}$-map and comparison with principally polarized abelian surfaces}
In this section we sketch the ideas that lead us to the construction of the $\mathbf{T}$-map in the previous section and allows us to produce a precise comparison with the GMCD method applied to principally polarized abelian surfaces in \cite{mov13,tiago}. The construction performed here is taken from \cite{nikulin}.

Let $L=\mathbb{Z}e_1\oplus \mathbb{Z}e_2\oplus\mathbb{Z}e_3\oplus \mathbb{Z}e_4$ be a free abelian group of rank $4$.

\begin{prop}
The scalar product $(\cdot,\cdot)$ on $L\wedge L$ defined by means of $$u\wedge v=-(u,v)e_1\wedge e_2\wedge e_3\wedge e_4$$ is an even unimodular integral symmetric bilinear form of signature $(3,3)$.
\end{prop}

\begin{proof}
Let $e_{ij}:=e_i\wedge e_j$. Then $\{e_{12},e_{13},e_{14},e_{23},e_{24}, e_{34}\}$. is a basis for $L\wedge L$. In this basis, the matrix of $(\cdot,\cdot)$ corresponds to 
\begin{equation*}
\begin{bmatrix}
0 & 0 & 0 & 0 & 0 & -1 \\
0 & 0 & 0 & 0 & 1 & 0 \\
0 & 0 & 0 & -1 & 0 & 0 \\
0 & 0 & -1 & 0 & 0 & 0 \\
0 & 1 & 0 & 0 & 0 & 0 \\
-1 & 0 & 0 & 0 & 0 & 0 
\end{bmatrix}
\end{equation*}
The proof follows from this.
\end{proof}

We have a well-defined group homomorphism 

\begin{equation*}
\wedge^2(\cdot): End_{\mathbb{Z}}(L)\rightarrow End_{\mathbb{Z}}(L\wedge L),
\end{equation*}

which restricts to a homomorphism 

\begin{equation}\label{isom}
\wedge^2(\cdot):\mathsf{SL}(L)\rightarrow \mathsf{O}(L\wedge L,(\cdot,\cdot)).
\end{equation}

In other words, $\mathsf{SL}(L)$ acts from the left, by isometries, on $L\wedge L$.

Now, we make $L$ into a symplectic lattice by defininig a symplectic form $J$ on $L$ by means of 
\begin{equation}
J(x,y)e_1\wedge e_2\wedge e_3\wedge e_4=x\wedge y \wedge (e_{13}+e_{24}).
\end{equation}

Therefore, $\mathsf{Sp}_4(\mathbb{Z})\cong \mathsf{O}(L,J)$. 

\begin{prop}\label{bender}
$\mathsf{Sp}_4(\mathbb{Z})\cong\{g\in\mathsf{SL}(L)|(\wedge^2g)(e_{13}+e_{24})=e_{13}+e_{24}\}.$
\end{prop}

\begin{proof}
By \cite{bender}, $\mathsf{Sp}_4(\mathbb{Z})$ is generated by 

\begin{equation*}
\begin{bmatrix}
1 & 0 & 0 & 0 \\
1 & -1 & 0 & 0 \\
0 & 0 & 1 & 1 \\
0 & 0 & 0 & -1 
\end{bmatrix},
\begin{bmatrix}
0 & 0 & -1 & 0 \\
0 & 0 & 0 & -1 \\
1 & 0 & 1 & 0 \\
0 & 1 & 0 & 0 
\end{bmatrix}.
\end{equation*}
The elements of $\mathsf{O}(L,J)$ associated to the previous symplectic matrices are readily seen to belong to $\{g\in\mathsf{SL}(L)|(\wedge^2g)(e_{13}+e_{24})=e_{13}+e_{24}\}.$ On the other hand, let $g\in \mathsf{SL}(L)$ such that $(\wedge^2g)(e_{13}+e_{24})=e_{13}+e_{24}$. Then, $gx\wedge gy \wedge (e_{13}+e_{24})=gx\wedge gy \wedge (\wedge^2g)(e_{13}+e_{24})=det(g)x\wedge y \wedge (e_{13}+e_{24})=x\wedge y \wedge (e_{13}+e_{24})$. This implies, by definition of $J$, that $g\in\mathsf{O}(L,J)$.
\end{proof}

As a corollary to the previous Proposition \ref{bender}, we get that the action in \eqref{isom} restricts to a left action by isometries of $\mathsf{Sp}_4(\mathbb{Z})$ on $M:=(e_{13}+e_{24})^{\perp}$. Therefore, we have a representation

\begin{equation}\label{comp1}
\wedge^2(\cdot):\mathsf{Sp}_4(\mathbb{Z})\rightarrow \mathsf{O}(M_,(\cdot,\cdot)).
\end{equation}

By \cite[Lemma 1.1.]{nikulin}, this homomorphism has kernel equal to $\{\pm I_4\}$ and its image can be defined as follows. Since $M$ is nondegenerate (see Proposition \ref{M}), we have a canonical monomorphism $M\rightarrow M^*$ which allows us to embed $M$ into $M^*$. There is a canonical surjective homomorphism $\mathsf{O}(M_,(\cdot,\cdot))\rightarrow M^*/M$. The kernel of this homomorphism is denoted by $\mathsf{O}_0(M_,(\cdot,\cdot))$, and it is equal to the image of homomorphism \eqref{comp}. Since $|M^*/M|=d(M)=2$, $\mathsf{O}_0(M_,(\cdot,\cdot))$ has index 2 in $\mathsf{O}(M_,(\cdot,\cdot))$.

Now, we begin to relate the previous construction to what was done at the beginning of this section.

\begin{prop}\label{M}
$M$ is isometric to $N^{\perp}$.
\end{prop}
\begin{proof}
A simple computation shows that $M$ is a free $\mathbb{Z}$-module with basis $\{e_{12},e_{14}, e_{13}-e_{24},e_{32}, e_{43}\}$. In this basis, the matrix of $(\cdot,\cdot)|_M$ is equal to $\Phi$.
\end{proof}

As a corollary of the previous proposition, we can write the Griffiths-Dolgachev period domain for $N$-polarized K3 surfaces $\mathsf{D}$ as 
\begin{equation}
\mathsf{D}=\{z\in\mathbb{P}(M_{\mathbb{C}})|\,\, (z,z)=0\land (z,\overline{z})>0\}.
\end{equation} 

A computation shows that $\mathsf{D}=\mathsf{D}^+\coprod \overline{\mathsf{D}^+}$. Here, $\mathsf{D}^+$ is biholomorphic to Kodaira's symmetric bounded domain $IV_3$, which turns out to be biholomorphic to $\mathbb{H}_2$. More explicitelly: 
\begin{equation}\mathsf{D}^+=\{[\tau_2^2-\tau_1\tau_3:\tau_3:\tau_2:\tau_1:1]\in \mathbb{P}(M_{\mathbb{C}})|\,\, \text{Im}(\tau_1)\text{Im}(\tau_3)>\text{Im}(\tau_2)^2\land\text{Im}(\tau_1)>0\}.
\end{equation}

Homomorphism \eqref{comp} allows us to identify the generalized period domains for principally polarized abelian surfaces and $N$-polarized K3 surfaces.

The following definitions are taken from \cite[$\S 4.1$]{mov13}.

Let us first define 

\begin{equation}
\Psi_{AbS}:=
\begin{bmatrix}
0 & -I_2 \\
I_2 & 0 \\
\end{bmatrix}.
\end{equation}

\begin{df}
The manifold of period matrices for principally polarized abelian surfaces is 

\begin{equation*}\mathsf{\Pi}_{AbS}=\{P\in \mathsf{GL}_4(\mathbb{C})|\, P^{T}\Psi_{AbS} P=\Psi_{AbS} \text{ and } (P^1)^T\Psi_{AbS}\overline{P^1}>0\}.
\end{equation*}
\end{df}

Let $\Xi:=diag(1,1,2,1,1)$. We have:

\begin{prop}\label{comparison}
There is a well-defined biholomorphism $$F:\mathsf{\Pi}_{AbS}\rightarrow \mathsf{\Pi},$$ $$F(P)=\Xi(\wedge^2(P))\Xi^{-1}.$$
\end{prop}
\begin{proof}
We observe first that $\mathsf{\Pi}_{AbS}\subset\mathsf{Sp}_4(\mathbb{C})$, $\mathsf{\Pi}\subset\mathsf{O}_0(M_{\mathbb{C}},(\cdot,\cdot))$ and $\Xi^{-1} \Phi \Xi^{-1}=\Psi.$ Let $P\in\Pi_{AbS}$. Then, $P\in\mathsf{Sp}_4(\mathbb{C})$, and homomorphism \eqref{comp} implies that that $\tilde{P}:=\wedge^2(P)\in\mathsf{O}(M_{\mathbb{C}},(\cdot,\cdot))$. This means $\tilde{P}^T\Phi\tilde{P}=\Phi$. Let $F(P):=\Xi\tilde{P}\Xi^{-1}$. Therefore, $$F(P)^{T}\Psi F(P)=\Xi^{-1}\tilde{P}^{T}\Xi\Xi^{-1}\Phi\Xi^{-1}\Xi\tilde{P}\Xi^{-1}=\Xi^{-1}\Phi\Xi^{-1}=\Psi.$$ We conclude that $F(P):=\tilde{P}^{-T}\in \mathsf{\Pi}$. Since homomorphism \eqref{comp} has image equal to $\mathsf{O}_0(M,(\cdot,\cdot))$, $F$ is surjective. We deal now with injectivity. Suppose $F(P_1)=F(P_2)$ for two distinct $P_1,P_2\in \mathsf{\Pi}_{AbS}$. Then, $P_2=-P_1$, which is a contradiction because every symplectic matrix has determinant equal to $1$. Hence, $F$ is bijective. Since $\mathsf{\Pi}_{AbS}$ and $\mathsf{\Pi}$ are both complex manifolds, $F$ is a biholomorphism.
\end{proof}

\begin{cor}
The map $F$ in in Proposition \ref{comparison} induces a biholomorphism between the generalized period domain for principally polarized abelian surfaces $\mathsf{U}_{AbS}$, and  the generalized period domain for $N$-polarized K3 surfaces $\mathsf{U}$.

Since in both contexts the generalized period map is a biholomorphism, we conclude that the moduli space of enhanced principally polarized abelian surfaces $\mathsf{T}_{AbS}$, and the moduli space of enhanced $N$-polarized K3 surfaces are $\mathsf{T}$ are biholomorphic.
\end{cor}

In \cite{mov13}, the $\mathbf{T}$-map for principally polarized abelian surfaces was defined to be:  
\begin{gather}\mathbf{T}_{AbS}:\mathbb{H}_2\rightarrow \mathsf{\Pi}_{AbS}, \notag\\
\tau\mapsto \begin{bmatrix}
\tau &-I_2\\
I_2 & 0 \\
\end{bmatrix}.\notag
\end{gather}

Our $\mathbf{T}$-map is the composition 

\begin{equation}\mathbb{H}_2\xrightarrow{\mathbf{T}_{AbS}} \mathsf{\Pi}_{AbS}\xrightarrow F \mathsf{\Pi}.
\end{equation}

This explains the construction of the $\mathbf{T}$-map defined in $\S 5.3$.

%% file: s6.tex
In this section we prove the existence of $dim(\mathcal{M})$ modular vector fields on the moduli space $\mathsf{T}$. These vector fields are the generalization to the context of this work of the Ramanujan vector field considered in the introduction. Furthermore, we construct algebraic vector fields associated to each element of the AMSY-Lie algebra. This result, via special geometry manipulations was already obtained in \cite[Theorem 2.1]{alimtt}.

In the case of $N$-polarized K3 surfaces, by pulling back regular functions on $\mathsf{T}$ to $\mathbb{H}_2$ by means of the $\mathsf{t}$-map constructed in the previous section, we will interpret these vector fields as derivations of meromorphic Siegel modular forms.

\subsection{Modular vector fields}

Let us begin by recalling the following results from Section 3.

\begin{df}
The Lie subalgebra of $\mathfrak{gl}_{22-k}(\mathbb{C})$ generated by $\mathsf{Lie}(\mathsf{G}_k)$ and $\{\mathfrak{g}^T|\, \mathfrak{g}\in\mathfrak{nil}(\mathsf{Lie}(\mathsf{G}_k))\}$ is called the AMSY-Lie algebra, and is denoted by $\mathfrak{G}_k$.
\end{df}

\begin{prop}
$\mathfrak{G}_k\cong \mathfrak{so}_{22-k}(\mathbb{C})$. In particular, $\mathfrak{G}_{18}\cong \mathfrak{sl}_2(\mathbb{C})\oplus\mathfrak{sl}_2(\mathbb{C})$ and $\mathfrak{G}_{17}\cong\mathfrak{sp}_4(\mathbb{C})$. 
\end{prop}

\begin{lem}\label{unique} For each $\mathfrak{g}\in \mathfrak{G}$, there is a unique vector field $\tilde{\mathsf{R}}_{\mathfrak{g}}\in\Theta_{\Gamma_{\mathbb{Z}}\backslash\mathsf{\Pi}}$ such that
\begin{equation}\label{modvector}
[dx_{ij}(\tilde{\mathsf{R}}_{\mathfrak{g}})]=[x_{ij}]\mathfrak{g}.
\end{equation}
\end{lem}

\begin{proof} Let $x_{ij}$ be the usual matrix coordinates for $\mathsf{Mat}_5(\mathbb{C})$, and let us define 
\begin{equation}
\tilde{\mathsf{R}}_{\mathfrak{g}}=\sum_{i,j=1}^5 c_{ij}\frac{\partial}{\partial x_{ij}}\in \Theta_{\mathsf{Mat}_5(\mathbb{C})}
\end{equation}
by means of equation \ref{modvector}. Therefore, $[c_{ij}]=[x_{ij}]\mathfrak{g}.$ Now, observe that $\mathfrak{g}\Psi$ is antisymmetric, which implies that, for each $P\in \mathsf{\Pi}$, $P^T\Psi(\mathsf{\tilde{R}}_{\mathfrak{g}})_P+(\mathsf{\tilde{R}}_{\mathfrak{g}})_P^T\Psi P=P^T\Psi P \mathfrak{g}+\mathfrak{g}^T P^T\Psi P=\Psi \mathfrak{g}+\mathfrak{g}^T\Psi=0$.. Therefore, $\tilde{\mathsf{R}}_{\mathfrak{g}}$ descends to a holomorphic vector field on $\mathsf{\Pi}$, which we also denote by $\tilde{\mathsf{R}}_{\mathfrak{g}}\in\Theta_{\mathsf{\Pi}}$. Since the action of $\Gamma$ on $\mathsf{\Pi}$ is given by left multiplication, its derivative is also given by left multiplication. Therefore, equation \ref{modvector} implies that the holomorphic vector fields  $\tilde{\mathsf{R}}_{\mathfrak{g}}$ are $\Gamma$-invariant, which concludes the proof.

\end{proof}

The previous lemma will be used in the proof of the following theorem by means of the generalized period map, which allows us to work in the generalized period domain $\Gamma\backslash\mathsf{\Pi},$ and transport the information back to the moduli space $\mathsf{T}$. 

\begin{teo}\label{bigteo}
For each $\mathfrak{g}\in \mathfrak{G}$, there exists a unique algebraic vector field $\mathsf{R}_{\mathfrak{g}}\in \Theta_{\mathsf{T}}$ such that
\begin{equation}
\nabla_{\mathsf{R}_{\mathfrak{g}}} \alpha=\mathfrak{g}^T\alpha.
\end{equation}
Here, $\alpha=S\omega$ as in the previous sections.

In particular, for each $\mathfrak{g}\in \{\mathfrak{g}^T|\, \mathfrak{g}\in\mathfrak{nil}(\mathsf{Lie}(\mathsf{G}_k))\}$, the associated vector fields $\mathsf{R}_{\mathfrak{g}}$ are modular vector fields in the sense of \cite[Definition 6.13]{mov20}.
\end{teo}

\begin{proof}
We are going to explicitly compute the vector fields in the statement of this theorem. Let us fix an arbitrary $\mathfrak{g}\in \mathfrak{G}$, and write $\mathsf{R}=\mathsf{R}_{\mathfrak{g}}.$ For simplicity in notation, we will also denote the variables $a,b,c,d$ by $a_k$ over an index set $I$. We begin by observing that, by Lemma \ref{unique}, the existence of an unique holomorphic vector field $\mathsf{R}$ on $\mathsf{T}$ satisfying Equation \eqref{modular} is already guaranteed. We prove that $\mathsf{R}$ is algebraic and defined over $\mathbb{Q}$. To do this we use the fact that the Gauss-Manin connection $\nabla$ is defined over $\mathbb{Q}$. More explicitly, let us write such an $\mathsf{R}$ in a Zariski chart $O$ of $\mathsf{T}$ as \begin{equation}
\mathsf{R}=\sum_{k\in I}a_k\frac{\partial }{\partial t_k}+\sum_{i,j} b_{ij}\frac{\partial}{\partial s_{ij}}, 
\end{equation} 
where $a_k$ and $b_{ij}$ are holomorphic functions on $O$. Now, we aim to prove that these functions are indeed regular functions on $O$ . From Equation \ref{modular1}, we obtain 
\begin{equation}\label{clear}
\mathsf{GM}_{\alpha}(\mathsf{R})=\mathfrak{g}^T.
\end{equation}
Recall that $\mathsf{GM}_{\alpha}$ is the matrix of the Gauss-Manin connection in the frame $\alpha$. Since $\Psi \mathfrak{g}+\mathfrak{g}^T\Psi=0$ and $\Psi \mathsf{GM}_{\alpha}(\mathsf{R})^T+\mathsf{GM}_{\alpha}(\mathsf{R})\Psi=0.$ We conclude that both $\mathsf{GM}_{\alpha}(\mathsf{R})\Psi$ and $\mathfrak{g}^T\Psi$ are antisymmetric. This implies that the system of linear equations \begin{equation}\label{linear}
\mathsf{GM}_{\alpha}(R)_{i5}=\mathfrak{g}^T_{i5}, \,\, i=2,3,4
\end{equation}
in the variables $(a_k)_{k\in I}$ is a linear combination of the equations 

\begin{equation}
\mathsf{GM}_{\alpha}(R)_{1i}=\mathfrak{g}^T_{1i}, \,\, i=2,3,4.
\end{equation}
Furthermore, it also implies that $\mathsf{GM}_{\alpha}(R)_{15}=\mathfrak{g}^T_{15}=0$. Now, we affirm that the linear system 
\begin{equation}\label{system}
\mathsf{GM}_{\alpha}(R)_{1i}=\mathfrak{g}^T_{1i}, \,\, i=1,2,3,4.
\end{equation} of $|I|$ equations in the variables $(a_k)_{k\in I}$, has a non singular coefficients matrix. Suppose the contrary. Then there is another solution $(a_k')_{k\in I}$ to the system \ref{system}. Then, the $(a_k')_{k\in I}$ also satisfy equations \ref{linear}, since they are linear combinations of equations \ref{system}. It is not hard to see that we can find $b_{ij}'$ such that \begin{equation}
\mathsf{R}'=\sum_{k\in I}a_k'\frac{\partial }{\partial t_k}+\sum_{i,j} b_{ij}'\frac{\partial}{\partial s_{ij}}, 
\end{equation} satisfies $\mathsf{GM}_{\alpha}(\mathsf{R}')=\mathfrak{g}^T$, since $\mathsf{GM}_{\alpha}=(dS+S\mathsf{GM}_{\omega})S^{-1}$ and $S^{-1}$ has a similar block matrix form as $S$.\footnote{The details for isolating the $b_{ij}'$ in terms of the $(a_k')_{k\in I}$ can be found in the appendix to this section.} This is a contradiction with the uniqueness in Lemma \ref{unique}. Therefore, the coefficient matrix of the system \ref{system} is nonsingular, which implies that the $(a_k)_{k\in I}$ are regular functions on $O$. An this implies that the $b_{ij}$ are regular functions too, since we can isolate them in terms of the $(a_k')_{k\in I}$.\footnote{See the previous footnote.} Therefore, $\mathsf{R}$ is algebraic in $O$. The same computation can be done in every chart, and they are compatible in the intersections by lemma \ref{unique}.
\end{proof}

\begin{cor}
The vector fields $\{\mathsf{R}_{\mathfrak{g}}\}_{\mathfrak{g}\in \mathfrak{G}}$ induce an infinite-dimensional representation of $\mathfrak{sp}_4(\mathbb{C})$.
\end{cor}

\begin{teo}
\begin{equation}\bigcap_{\mathfrak{g}\in \mathsf{Lie}(\mathfrak{G}),\mathfrak{g}\neq\mathfrak{g}_0}Ker(\mathsf{R}_{\mathfrak{g}})=\mathbb{C}[a,b,c,d]
\end{equation}
\end{teo}
\begin{proof}
Let us observe that every $\mathsf{R}_{\mathfrak{g}}$, with $\mathfrak{g}\in \mathsf{Lie}(\mathfrak{G})$ and $\mathfrak{g}\neq\mathfrak{g}_0$, is of the form $\sum_{i,j} b_{ij}\frac{\partial}{\partial s_{ij}}$ in any Zariski chart $O$ of $\mathsf{T}$ as in the proof of the previous theorem. This follows from the fact that, for such a $\mathfrak{g}$, system \ref{system} is homogeneous. This implies the containment ($\supseteq$). For the containment ($\subseteq$), let $\mathsf{s}_1,\ldots,\mathsf{s}_6$ be algebraically independent variables among the variables $s_{ij}$, and let us consider the $\mathcal{O}_{\mathsf{T}}$-span of the 6 vectors fields $\mathsf{R}_{\mathfrak{g}}$, with $\mathfrak{g}\in \mathsf{Lie}(\mathfrak{G})$ and $\mathfrak{g}\neq\mathfrak{g}_0$. Since they are linearly independent over $\mathcal{O}_{\mathsf{T}}$ 
, the span of such vector fields forms a 6-dimensional $\mathcal{O}_{\mathsf{T}}$-submodule of the $\mathcal{O}_{\mathsf{T}}$-module generated by the span of the 6 linearly independent vector fields $\frac{\partial}{\partial \mathsf{s}_1},\ldots,\frac{\partial}{\partial \mathsf{s}_6}$. Passing to the field of quotients of $H^0(O,\mathcal{O})$, we get that for every $i$ and $\mathfrak{g}$ as before, there are polynomials $P_i,P^{\mathfrak{g}}_{i}\in H^0(O,\mathcal{O})$ such that $P_i\frac{\partial}{\partial\mathsf{s}_i}=\sum_{\mathfrak{g}\in \mathsf{Lie}(\mathfrak{G}),\mathfrak{g}\neq\mathfrak{g}_0}P_i^{\mathfrak{g}}\mathsf{R}_{\mathfrak{g}}$. Let $Q\in \bigcap_{\mathfrak{g}\in \mathsf{Lie}(\mathfrak{G}),\mathfrak{g}\neq\mathfrak{g}_0}Ker(\mathsf{R}_{\mathfrak{g}})$. Then, $P_i\frac{\partial Q}{\partial\mathsf{s}_i}=\sum_{\mathfrak{g}\in \mathsf{Lie}(\mathfrak{G}),\mathfrak{g}\neq\mathfrak{g}_0}P_i^{\mathfrak{g}}\mathsf{R}_{\mathfrak{g}}(Q)=0$ for each $i$. Linear independency of the $\mathsf{R}_{\mathfrak{g}}$ implies that $\frac{\partial Q}{\partial\mathsf{s}_i}=0$ for each $i$. This implies $Q\in \mathbb{C}[a_k]_{k\in I}$.    

\end{proof} 

The previous theorem in the case of Calabi-Yau threefolds was used to define the ambiguity.

\subsection{Meromorphic Siegel Quasimodular forms}
In this section, we specialize the previous results to $N$-polarized K3 surfaces.

\begin{df}
The algebra of algebraic Siegel quasimodular forms $\mathfrak{M}^{alg}(\mathsf{Sp}_4(\mathbb{C}))$  of genus two for the group $\mathsf{Sp}_4(\mathbb{C})$ is defined to be the ring $H^0(\mathsf{T},\mathcal{O}_{\mathsf{T}})$ of regular global sections on $\mathsf{T}$.

Since each vector field $\mathsf{R}_{\mathfrak{g}}$ of Theorem \ref{bigteo} can be seen as a derivation of $\mathfrak{M}^{alg}(\mathsf{Sp}_4(\mathbb{C}))$, we have that $(\mathfrak{M}^{alg}(\mathsf{Sp}_4(\mathbb{C})),\{\mathsf{R}_{\mathfrak{g}}\}_{g\in\mathfrak{G}})$ is a differential algebra. We call it the differential algebra of algebraic Siegel quasimodilar forms of genus two.
\end{df}

The to pass from algebraic Siegel quasimodular forms to quasimodularforms as holomorphic functions on the Siegel half-plane $\mathbb{H}_2$ we use the $\mathsf{t}$-map constructed in the previous section 

\begin{df}
The algebra of meromorphic Siegel quasimodular forms is defined to be the pullback $\mathsf{t}^*(\mathcal{O}_{\mathsf{T}})$.
\end{df}

Since $\mathbb{C}[a,b,c,d]$ is $\mathbb{C}$-flat, we have a natural monomorphism $\mathbb{C}[a,b,c,d]\rightarrow \mathcal{O}_{\mathsf{T}}$. Therefore, we can assume that $\mathbb{C}[a,b,c,d]\subset \mathcal{O}_{\mathsf{T}}.$

\begin{teo}
The algebra $\mathsf{t}^*(\mathbb{C}[a,b,c,d])$ coincides with the algebra of Siegel modular forms of even weight.
\end{teo}
\begin{proof}
Up to constants, $t^*(a)=E_4,t^*(b)=E_6,t^*(c)=\chi_{10},t^*(d)=\chi_{12}$. Then, the theorem follows from a classical result from Igusa \cite{igusa}. 
\end{proof}

%% file: s7.tex
The vanishing locus is $$\Sigma=Z(\Delta_{f})\cup Z(c_{1,0,3}) \cup Z(c_{1,0,2}) \cup Z(b_{1,0,1})\cup Z(c_{1,0,1}) \cup Z(c_{0,0,2})\cup Z(s) \cup Z(a) \subset \mathbb{C}^5$$

To explain the previous notations, recall that $$\mathcal{B}=\{xw^3,xw^2.w^3,xy,xw,w^2,x,y,w,1\},$$ and we are indentifying each monomial with its tuple of exponents.

$a_{i,j,k}=\frac{b_{i,j,k}}{c_{i,j,k}}$

$b_{1,0,3}=\frac{-41278242816}{s^4}\Delta_{f}$

$b_{1,0,2}=sc_{1,0,3}$

$b_{0,0,3}=12sc_{1,0,2}$

$b_{0,0,2}=\frac{-1}{18}c_{1,0,1}$

$a_{1,1,0}=-24s$

$b_{0,0,2}=\frac{-1}{18}c_{1,0,1}$

$a_{0,1,0}=-24s$

$b_{1,0,0}=\frac{1}{18}c_{0,0,2}$

$a_{0,0,1}=12a$

$a_{0,0,0}=12a$

$\Delta_{f}  =\frac{1}{41278242816}(29386561536a^{12}s^{8}+2448880128a^{10}c^{2}ds^{7}-117546246144a^{9}b^{2}s^{8}+1360488960a^{9}bc^{3}s^{7}+629856a^{9}c^{6}s^{6}+1088391168a^{9}d^{3}s^{7}-78364164096a^{9}ds^{8}+6530347008a^{8}bcd^{2}s^{7}-156728328192a^{8}bcs^{8}+43460064a^{8}c^{4}d^{2}s^{6}+204073344a^{8}c^{4}s^{7}-816293376a^{7}b^{2}c^{2}ds^{7}+41570496a^{7}bc^{5}ds^{6}+34992a^{7}c^{8}ds^{5}-45349632a^{7}c^{2}d^{4}s^{6}-4534963200a^{7}c^{2}d^{2}s^{7}-30474952704a^{7}c^{2}s^{8}+176319369216a^{6}b^{4}s^{8}-1904684544a^{6}b^{3}c^{3}s^{7}+6298560a^{6}b^{2}c^{6}s^{6}-1088391168a^{6}b^{2}d^{3}s^{7}+78364164096a^{6}b^{2}ds^{8}+23328a^{6}bc^{9}s^{5}-25194240a^{6}bc^{3}d^{3}s^{6}-11972302848a^{6}bc^{3}ds^{7}-11664a^{6}c^{6}d^{3}s^{5}+4094064a^{6}c^{6}ds^{6}+10077696a^{6}d^{6}s^{6}-2539579392a^{6}d^{4}s^{7}+69657034752a^{6}d^{2}s^{8}-69657034752a^{6}s^{9}-13060694016a^{5}b^{3}cd^{2}s^{7}+313456656384a^{5}b^{3}cs^{8}-86920128a^{5}b^{2}c^{4}d^{2}s^{6}-5487305472a^{5}b^{2}c^{4}s^{7}-1154736a^{5}bc^{7}s^{6}-17414258688a^{5}bcd^{3}s^{7}+139314069504a^{5}bcds^{8}+3888a^{5}c^{10}s^{5}-92798784a^{5}c^{4}d^{3}s^{6}-2509346304a^{5}c^{4}ds^{7}-5714053632a^{4}b^{4}c^{2}ds^{7}-83140992a^{4}b^{3}c^{5}ds^{6}-69984a^{4}b^{2}c^{8}ds^{5}+90699264a^{4}b^{2}c^{2}d^{4}s^{6}-39544879104a^{4}b^{2}c^{2}d^{2}s^{7}+217678233600a^{4}b^{2}c^{2}s^{8}-255301632a^{4}bc^{5}d^{2}s^{6}-1894606848a^{4}bc^{5}s^{7}-71928a^{4}c^{8}d^{2}s^{5}-532170a^{4}c^{8}s^{6}+92378880a^{4}c^{2}d^{5}s^{6}+886837248a^{4}c^{2}d^{3}s^{7}+14511882240a^{4}c^{2}ds^{8}-117546246144a^{3}b^{6}s^{8}-272097792a^{3}b^{5}c^{3}s^{7}-14486688a^{3}b^{4}c^{6}s^{6}-1088391168a^{3}b^{4}d^{3}s^{7}+78364164096a^{3}b^{4}ds^{8}-46656a^{3}b^{3}c^{9}s^{5}+50388480a^{3}b^{3}c^{3}d^{3}s^{6}-21042229248a^{3}b^{3}c^{3}ds^{7}+23328a^{3}b^{2}c^{6}d^{3}s^{5}-187802064a^{3}b^{2}c^{6}ds^{6}-20155392a^{3}b^{2}d^{6}s^{6}-3627970560a^{3}b^{2}d^{4}s^{7}-417942208512a^{3}b^{2}s^{9}-184680a^{3}bc^{9}ds^{5}+230107392a^{3}bc^{3}d^{4}s^{6}-6691590144a^{3}bc^{3}d^{2}s^{7}+52565262336a^{3}bc^{3}s^{8}+23328a^{3}c^{6}d^{4}s^{5}-46492704a^{3}c^{6}d^{2}s^{6}-167028480a^{3}c^{6}s^{7}-20155392a^{3}d^{7}s^{6}+1773674496a^{3}d^{5}s^{7}-24508956672a^{3}d^{3}s^{8}+92876046336a^{3}ds^{9}+6530347008a^{2}b^{5}cd^{2}s^{7}-156728328192a^{2}b^{5}cs^{8}+43460064a^{2}b^{4}c^{4}d^{2}s^{6}-3423897216a^{2}b^{4}c^{4}s^{7}-31597776a^{2}b^{3}c^{7}s^{6}-17414258688a^{2}b^{3}cd^{3}s^{7}+139314069504a^{2}b^{3}cds^{8}-97200a^{2}b^{2}c^{10}s^{5}+19735488a^{2}b^{2}c^{4}d^{3}s^{6}-5129547264a^{2}b^{2}c^{4}ds^{7}+46656a^{2}bc^{7}d^{3}s^{5}-57386880a^{2}bc^{7}ds^{6}-40310784a^{2}bcd^{6}s^{6}+8062156800a^{2}bcd^{4}s^{7}-38698352640a^{2}bcd^{2}s^{8}-185752092672a^{2}bcs^{9}-37125a^{2}c^{10}ds^{5}+93172032a^{2}c^{4}d^{4}s^{6}-535237632a^{2}c^{4}d^{2}s^{7}+4434186240a^{2}c^{4}s^{8}+4081466880ab^{6}c^{2}ds^{7}+41570496ab^{5}c^{5}ds^{6}+34992ab^{4}c^{8}ds^{5}-45349632ab^{4}c^{2}d^{4}s^{6}-8162933760ab^{4}c^{2}d^{2}s^{7}-47889211392ab^{4}c^{2}s^{8}-84820608ab^{3}c^{5}d^{2}s^{6}-1229478912ab^{3}c^{5}s^{7}-68040ab^{2}c^{8}d^{2}s^{5}-11080800ab^{2}c^{8}s^{6}+89019648ab^{2}c^{2}d^{5}s^{6}+2660511744ab^{2}c^{2}d^{3}s^{7}+28056305664ab^{2}c^{2}ds^{8}-33750abc^{11}s^{5}+81368064abc^{5}d^{3}s^{6}-268738560abc^{5}ds^{7}+48600ac^{8}d^{3}s^{5}-4536000ac^{8}ds^{6}-57106944ac^{2}d^{6}s^{6}+510603264ac^{2}d^{4}s^{7}+2794881024ac^{2}d^{2}s^{8}-25798901760ac^{2}s^{9}+29386561536b^{8}s^{8}+816293376b^{7}c^{3}s^{7}+7558272b^{6}c^{6}s^{6}+1088391168b^{6}d^{3}s^{7}-78364164096b^{6}ds^{8}+23328b^{5}c^{9}s^{5}-25194240b^{5}c^{3}d^{3}s^{6}-1813985280b^{5}c^{3}ds^{7}-11664b^{4}c^{6}d^{3}s^{5}-15326496b^{4}c^{6}ds^{6}+10077696b^{4}d^{6}s^{6}-2539579392b^{4}d^{4}s^{7}+69657034752b^{4}d^{2}s^{8}-69657034752b^{4}s^{9}-48600b^{3}c^{9}ds^{5}+52068096b^{3}c^{3}d^{4}s^{6}+1209323520b^{3}c^{3}d^{2}s^{7}-6127239168b^{3}c^{3}s^{8}+23328b^{2}c^{6}d^{4}s^{5}+7931520b^{2}c^{6}d^{2}s^{6}-130636800b^{2}c^{6}s^{7}-20155392b^{2}d^{7}s^{6}+1773674496b^{2}d^{5}s^{7}-24508956672b^{2}d^{3}s^{8}+92876046336b^{2}ds^{9}+27000bc^{9}d^{2}s^{5}-1080000bc^{9}s^{6}-28366848bc^{3}d^{5}s^{6}-310542336bc^{3}d^{3}s^{7}+4299816960bc^{3}ds^{8}-3125c^{12}s^{5}-11664c^{6}d^{5}s^{5}+3283200c^{6}d^{3}s^{6}+31104000c^{6}ds^{7}+10077696d^{8}s^{6}-322486272d^{6}s^{7}+3869835264d^{4}s^{8}-20639121408d^{2}s^{9}+41278242816s^{10}).$

$c_{1,0,3}=7558272a^{10}c^{3}s+52907904a^{9}cd^{2}s-45349632a^{9}cs^{2}+90699264a^{8}bc^{2}ds+314928a^{8}c^{5}d+7558272a^{7}b^{2}c^{3}s+209952a^{7}bc^{6}+30233088a^{7}bd^{3}s-181398528a^{7}bds^{2}-52482519424a^{8}d^{2}s+5038848a^{7}bcds+34992a^{7}c^{4}d+2519424a^{6}b^{2}c^{2}s+23328a^{6}bc^{5}-11664a^{6}c^{2}d^{3}+909792a^{6}c^{2}ds-5038848a^{5}b^{2}d^{2}s+419904a^{5}bc^{3}s+3888a^{5}c^{6}-5038848a^{5}d^{3}s-13436928a^{5}ds^{2}-10077696a^{4}b^{3}cds-69984a^{4}b^{2}c^{4}d-20155392a^{4}bcd^{2}s-10077696a^{4}bcs^{2}-71928a^{4}c^{4}d^{2}-54432a^{4}c^{4}s-5038848a^{3}b^{4}c^{2}s-46656a^{3}b^{3}c^{5}+23328a^{3}b^{2}c^{2}d^{3}-26593920a^{3}b^{2}c^{2}ds-184680a^{3}bc^{5}d+23328a^{3}c^{2}d^{4}-3786912a^{3}c^{2}d^{2}s-2052864a^{3}c^{2}s^{2}+2519424a^{2}b^{4}d^{2}s-10497600a^{2}b^{3}c^{3}s-97200a^{2}b^{2}c^{6}-5038848a^{2}b^{2}d^{3}s-26873856a^{2}b^{2}ds^{2}+46656a^{2}bc^{3}d^{3}-8654688a^{2}bc^{3}ds-37125a^{2}c^{6}d+2519424a^{2}d^{4}s+13436928a^{2}d^{2}s^{2}+17915904a^{2}s^{3}+5038848ab^{5}cds+34992ab^{4}c^{4}d-10077696ab^{3}cd^{2}s-30233088ab^{3}cs^{2}-68040ab^{2}c^{4}d^{2}-3917160ab^{2}c^{4}s-33750abc^{7}+4945536abcd^{3}s+14183424abcds^{2}+48600ac^{4}d^{3}-820800ac^{4}ds+2519424b^{6}c^{2}s+23328b^{5}c^{5}-11664b^{4}c^{2}d^{3}-4548960b^{4}c^{2}ds-48600b^{3}c^{5}d+23328b^{2}c^{2}d^{4}+1687392b^{2}c^{2}d^{2}s-10264320b^{2}c^{2}s^{2}+27000bc^{5}d^{2}-432000bc^{5}s-3125c^{8}-11664c^{2}d^{5}+311040c^{2}d^{3}s+6220800c^{2}ds^{2}80a^{7}c^{3}d^{3}-4199040a^{7}c^{3}ds-68024448a^{6}b^{2}cd^{2}s-45349632a^{6}b^{2}cs^{2}-279936a^{6}bc^{4}d^{2}-26453952a^{6}bc^{4}s+34992a^{6}c^{7}+139968a^{6}cd^{5}-114213888a^{6}cd^{3}s-181398528a^{5}b^{3}c^{2}ds-629856a^{5}b^{2}c^{5}d-387991296a^{5}bc^{2}d^{2}s+60466176a^{5}bc^{2}s^{2}-694008a^{5}c^{5}d^{2}-6526008a^{5}c^{5}s-37791360a^{4}b^{4}c^{3}s-419904a^{4}b^{3}c^{6}-60466176a^{4}b^{3}d^{3}s+362797056a^{4}b^{3}ds^{2}+1049760a^{4}b^{2}c^{3}d^{3}-435020544a^{4}b^{2}c^{3}ds-1662120a^{4}bc^{6}d-60466176a^{4}bd^{4}s+241864704a^{4}bd^{2}s^{2}+1934917632a^{4}bs^{3}+1073088a^{4}c^{3}d^{4}-58366656a^{4}c^{3}d^{2}s+1119744a^{4}c^{3}s^{2}-22674816a^{3}b^{4}cd^{2}s+226748160a^{3}b^{4}cs^{2}+559872a^{3}b^{3}c^{4}d^{2}-93218688a^{3}b^{3}c^{4}s-874800a^{3}b^{2}c^{7}-279936a^{3}b^{2}cd^{5}-285534720a^{3}b^{2}cd^{3}s+725594112a^{3}b^{2}cds^{2}+2636064a^{3}bc^{4}d^{3}-125481312a^{3}bc^{4}ds-334125a^{3}c^{7}d-279936a^{3}cd^{6}+69144192a^{3}cd^{4}s-49268736a^{3}cd^{2}s^{2}+698720256a^{3}cs^{3}+90699264a^{2}b^{5}c^{2}ds+314928a^{2}b^{4}c^{5}d-337602816a^{2}b^{3}c^{2}d^{2}s+423263232a^{2}b^{3}c^{2}s^{2}+554040a^{2}b^{2}c^{5}d^{2}-28815912a^{2}b^{2}c^{5}s-303750a^{2}bc^{8}-559872a^{2}bc^{2}d^{5}+90139392a^{2}bc^{2}d^{3}s+282175488a^{2}bc^{2}ds^{2}+882900a^{2}c^{5}d^{3}-10847520a^{2}c^{5}ds+22674816ab^{6}c^{3}s+209952ab^{5}c^{6}+30233088ab^{5}d^{3}s-181398528ab^{5}ds^{2}-524880ab^{4}c^{3}d^{3}-44509824ab^{4}c^{3}ds-437400ab^{3}c^{6}d-60466176ab^{3}d^{4}s+241864704ab^{3}d^{2}s^{2}+1934917632ab^{3}s^{3}+1026432ab^{2}c^{3}d^{4}-32052672ab^{2}c^{3}d^{2}s+148925952ab^{2}c^{3}s^{2}+648000abc^{6}d^{2}-1814400abc^{6}s+30233088abd^{5}s-53747712abd^{3}s^{2}-1504935936abds^{3}-28125ac^{9}-688176ac^{3}d^{5}+11010816ac^{3}d^{3}s+83358720ac^{3}ds^{2}+37791360b^{6}cd^{2}s-136048896b^{6}cs^{2}-279936b^{5}c^{4}d^{2}-1259712b^{5}c^{4}s+139968b^{4}cd^{5}-83980800b^{4}cd^{3}s+241864704b^{4}cds^{2}+583200b^{3}c^{4}d^{3}+2519424b^{3}c^{4}ds-279936b^{2}cd^{6}+54027648b^{2}cd^{4}s-201553920b^{2}cd^{2}s^{2}+591224832b^{2}cs^{3}-324000bc^{4}d^{4}-5495040bc^{4}d^{2}s+24883200bc^{4}s^{2}+37500c^{7}d^{2}+180000c^{7}s+139968cd^{7}-7838208cd^{5}s+98537472cd^{3}s^{2}-358318080cds^{3}.$

$c_{1,0,2}=2519424a^{8}d^{2}s+5038848a^{7}bcds+34992a^{7}c^{4}d+2519424a^{6}b^{2}c^{2}s+23328a^{6}bc^{5}-11664a^{6}c^{2}d^{3}+909792a^{6}c^{2}ds-5038848a^{5}b^{2}d^{2}s+419904a^{5}bc^{3}s+3888a^{5}c^{6}-5038848a^{5}d^{3}s-13436928a^{5}ds^{2}-10077696a^{4}b^{3}cds-69984a^{4}b^{2}c^{4}d-20155392a^{4}bcd^{2}s-10077696a^{4}bcs^{2}-71928a^{4}c^{4}d^{2}-54432a^{4}c^{4}s-5038848a^{3}b^{4}c^{2}s-46656a^{3}b^{3}c^{5}+23328a^{3}b^{2}c^{2}d^{3}-26593920a^{3}b^{2}c^{2}ds-184680a^{3}bc^{5}d+23328a^{3}c^{2}d^{4}-3786912a^{3}c^{2}d^{2}s-2052864a^{3}c^{2}s^{2}+2519424a^{2}b^{4}d^{2}s-10497600a^{2}b^{3}c^{3}s-97200a^{2}b^{2}c^{6}-5038848a^{2}b^{2}d^{3}s-26873856a^{2}b^{2}ds^{2}+46656a^{2}bc^{3}d^{3}-8654688a^{2}bc^{3}ds-37125a^{2}c^{6}d+2519424a^{2}d^{4}s+13436928a^{2}d^{2}s^{2}+17915904a^{2}s^{3}+5038848ab^{5}cds+34992ab^{4}c^{4}d-10077696ab^{3}cd^{2}s-30233088ab^{3}cs^{2}-68040ab^{2}c^{4}d^{2}-3917160ab^{2}c^{4}s-33750abc^{7}+4945536abcd^{3}s+14183424abcds^{2}+48600ac^{4}d^{3}-820800ac^{4}ds+2519424b^{6}c^{2}s+23328b^{5}c^{5}-11664b^{4}c^{2}d^{3}-4548960b^{4}c^{2}ds-48600b^{3}c^{5}d+23328b^{2}c^{2}d^{4}+1687392b^{2}c^{2}d^{2}s-10264320b^{2}c^{2}s^{2}+27000bc^{5}d^{2}-432000bc^{5}s-3125c^{8}-11664c^{2}d^{5}+311040c^{2}d^{3}s+6220800c^{2}ds^{2}.$

$b_{1,0,1}=-(34992a^{7}c^{3}d+23328a^{6}bc^{4}-11664a^{6}cd^{3}+559872a^{6}cds+3888a^{5}c^{5}-69984a^{4}b^{2}c^{3}d-71928a^{4}c^{3}d^{2}-62208a^{4}c^{3}s-46656a^{3}b^{3}c^{4}+23328a^{3}b^{2}cd^{3}-1119744a^{3}b^{2}cds-184680a^{3}bc^{4}d+23328a^{3}cd^{4}-1026432a^{3}cd^{2}s-1306368a^{3}cs^{2}-97200a^{2}b^{2}c^{5}+46656a^{2}bc^{2}d^{3}-2426112a^{2}bc^{2}ds-37125a^{2}c^{5}d+34992ab^{4}c^{3}d-68040ab^{2}c^{3}d^{2}-311040ab^{2}c^{3}s-33750abc^{6}-2239488abds^{2}+48600ac^{3}d^{3}-583200ac^{3}ds+23328b^{5}c^{4}-11664b^{4}cd^{3}+559872b^{4}cds-48600b^{3}c^{4}d+23328b^{2}cd^{4}-1026432b^{2}cd^{2}s-933120b^{2}cs^{2}+27000bc^{4}d^{2}-108000bc^{4}s-3125c^{7}-11664cd^{5}+466560cd^{3}s).$

$c_{1,0,1}=23328a^{6}c^{2}-46656a^{3}b^{2}c^{2}-43416a^{3}c^{2}d-77760a^{2}bc^{3}-7776abcd^{2}-15750ac^{4}+23328b^{4}c^{2}-42120b^{2}c^{2}d+16200c^{2}d^{2}.$

$c_{0,0,2}=126a^{3}c+216abd+90b^{2}c.$

$c_{1,0,0}=2a^2.$